\documentclass{amsart}
\usepackage{amsmath,amsfonts,amssymb}
\usepackage{graphicx}

\newtheorem{theorem}{Theorem}[section]
\newtheorem{lemma}[theorem]{Lemma}

\newtheorem{cor}[theorem]{Corollary}

\theoremstyle{definition}

\theoremstyle{remark}
\newtheorem{remark}[theorem]{Remark}

\numberwithin{equation}{section}

\begin{document}

\title[On the Constants and Extremal Function and Sequence for Hardy Inequalities in $L_p$ and $l_p$]{On the Constants and Extremal Function and Sequence for Hardy Inequalities in $L_p$ and $l_p$}
\thanks{Research supported by the Bulgarian National Research Fund through Contract KP-06-N62/4.}

\author[Ivan Gadjev]{Ivan Gadjev}
 \address{Department of Mathematics and Informatics,
University of Sofia,
5 James Bourchier Blvd., 
1164 Sofia, Bulgaria}
\email{gadjev@fmi.uni-sofia.bg}

      

\begin{abstract}
  We study the behavior of the smallest possible constants $d(a,b)$ and $d_n$   in Hardy inequalities
$$
\int_a^b\left(\frac{1}{x}\int_a^xf(t)dt\right)^p\,dx\leq
d(a,b)\,\int_a^b [f(x)]^p dx
$$
and
$$
\sum_{k=1}^{n}\Big(\frac{1}{k}\sum_{j=1}^{k}a_j\Big)^p\leq
d_n\,\sum_{k=1}^{n}a_k^p.
$$
The exact rate of convergence of $d(a,b)$ and $d_n$ is established and 
the ``almost extremal'' function and sequence are found.
\end{abstract}

\maketitle

\section{Introduction and statement of the results}

Between 1919 and 1925, in the series of papers  \cite {H1919, H1920, H1925} G. H. Hardy established the following inequalities
which are nowadays known as the celebrated Hardy's ones. Let $p>1$, then the integral inequality states that 
\begin{equation}\label{eq01}
\int_0^\infty\left(\frac{1}{x}\int_0^xf(t)dt\right)^p\,dx\leq
\left(\frac{p}{p-1}\right)^{p}\,\int_0^\infty f^p(x)dx
\end{equation}
holds for every $f$, such that $f(x)\geq0$ for $x\in (0,\infty)$ and $f^p$ is integrable over $(0,\infty)$.

The corresponding discrete version claims that 
\begin{equation}
\label{eq02}
\sum_{k=1}^{\infty}\Big(\frac{1}{k}\sum_{j=1}^{k}a_j\Big)^p\leq
\left(\frac{p}{p-1}\right)^{p}\,\sum_{k=1}^{\infty}{a_k^p}
\end{equation}
for every sequence $\{ a_k \}$ of non-negative numbers, for which the series on the right-hand side converges. 

In his 1920 paper \cite{H1920} Hardy claimed that he and Marcel Riesz derived (\ref{eq02}) independently, but in their results the 
larger constant $(p^2/(p-1))^p$ appears on the right-hand side.
E. Landau, in the letter \cite{Lan2} dated 1921, published later in \cite{Lan1}, 
was the first to establish  (\ref{eq02}) with the exact constant $(p/(p-1))^{p}$ in the sense that 
there is no smaller one for which \eqref{eq02}  holds for every sequence of non-negative numbers $a_k$. 
For the latter statement he considered the sequence $a_k^\ast=k^{-1/p-\varepsilon}$, suggested earlier by Hardy, and showed that 
$$
\left( \frac{a_1^\ast + \cdots + a_k^\ast}{k} \right)^p > \left(\frac{p}{p-1}\right)^p \left( (a_k^\ast)^p - \frac{p}{k^{2-1/p}} \right).
$$
Since $\sum_{k=1}^\infty (a_k^\ast)^p \to \infty$ as $\varepsilon \to 0$, the summation of the latter inequalities implies the sharpness of 
$(p/(p-1))^{p}$ for \eqref{eq02}.
In the same letter Landau pointed out that equality in \eqref{eq02} occurs only for the trivial sequence, that is,
 when $a_k=0$ for every $k \in \mathbb{N}$. Similarly, equality in \eqref{eq01} occurs if and only
if $f(x)\equiv0$ almost everywhere.  

The lack of nontrivial extremizers and the fact that the above argument of Landau does not work for finite sequences motivates one to consider, 
for any $a$ and $b$ with $-\infty\le  a < b\le \infty$ and weight positive a.e. functions $u(x), v(x)$, the so-called general Hardy's integral inequality 
\begin{equation}\label{eqI159}
\int_a^b\left(\int_a^xf(t)dt\right)^pu(x)\,dx\leq
d(a,b)\,\int_a^b f^p(x)v(x) dx,\quad f \in L^p[v;a,b]
\end{equation}
and its discrete counterpart 
\begin{equation}\label{eqI169}
\sum_{k=1}^{n}\Big(\sum_{j=1}^{k}a_j\Big)^pu_k\leq
d_n\,\sum_{k=1}^{n}{a_k^p}v_k, \qquad a_k,u_k,v_k\geq0,\,\,\,k=1,2,...,n.
\end{equation}
Obviously the equations \eqref{eqI159} and \eqref{eqI169} are the``finite versions'' of  \eqref{eq01} and \eqref{eq02}. The natural questions are: what are the best
constants $d(a,b)$ and $d_n$ and the corresponding extre\-mi\-zers. 
This is exactly our endeavor 
in this paper, mainly because of the importance of Hardy's inequalities, their far reaching generalizations, especially 
to the so-called Hardy-Sobolev inequalities, and thus the necessity of understanding them more thoroughly.
Answering the above questions in satisfactory manner
 for arbitrary weight functions and sequences such that the inequalities \eqref{eqI159} and \eqref{eqI169} hold, is impossible.
In this paper we consider the important "unweighted" versions of inequalities \eqref{eqI159} and \eqref{eqI169}, i.e.
\begin{equation}\label{eqI15}
\int_a^b\left(\frac{1}{x}\int_a^xf(t)dt\right)^p\,dx\leq
d(a,b)\,\int_a^b [f(x)]^p dx, \qquad f(x)\geq0,\ \ f \in L^p[a,b]
\end{equation}
and 
\begin{equation}\label{eqI16}
\sum_{k=1}^{n}\Big(\frac{1}{k}\sum_{j=1}^{k}a_j\Big)^p\leq
d_n\,\sum_{k=1}^{n}{a_k^p}, \qquad a_k\geq0,\,\,\,k=1,2,...,n.
\end{equation}

The behavior of the constant $d(a,b)$  was studied in many papers - see, for instance, \cite{Tom}, \cite{Tal},\cite{Muc}. The best results about $d(a,b)$ for $p>1$ could be summarized in the following way
(see, for instance \cite{KMP2007} or \cite{KP2003}).
Let
\[
B=\sup_{a<x<b} \left\{(x-a)^{p-1}\left(x^{1-p}-b^{1-p} \right)\right\}
\]
then for the constant $d(a,b)$ the next estimations are true
\[
\frac{1}{p-1}B\le d(a,b) \le \left(\frac{p}{p-1} \right)^p B.
\]
It is easy to see that only the right estimation gives asymptotically (when $a\rightarrow 0$ or $b\rightarrow \infty$ or both) the exact constant. But not the rate of convergence. 

In \cite{DGM} we studied the inequality \eqref{eqI15} for $p=2$ and established the exact constant $d(a,b)$ and the extremal function.

\begin{theorem}\cite{DGM}\label{th01}
Let  $a$ and $b$ be any fixed numbers with $0<a<b<\infty$. Then the inequality
\begin{equation}\label{MR01}
\int_a^b\left(\frac{1}{x}\int_a^xf(t)dt\right)^2 dx\leq
\frac{4}{1+4\alpha^2}\,\int_a^bf^2(x)\, dx,
\end{equation}
where $\alpha$ is the only solution of the equation
$$
\tan\left(\alpha\ln\frac{b}{a} \right)+2\alpha=0\ \ \mathrm{in\ the\ interval}\ \ \left(\frac{\pi}{2\ln\frac{b}{a}},\frac{\pi}{\ln\frac{b}{a}} \right),
$$
holds for every $f \in L^2[a,b]$.
Moreover, equality in \eqref{MR01} is attained for 
\[
f_{a,b}(x)=x^{-1/2}\left(2\alpha\cos(\alpha\ln x)+\sin(\alpha\ln x) \right).
\]
\end{theorem}

The behavior of the constant $d_n$ for $p=2$ as a function of $n$ was also studied extensively - see, for instance, \cite{HW1}, \cite{HW2},\cite{Wilf}, \cite{Wilf2}, \cite{Pec}. 
In \cite{Wilf} Herbert S. Wilf established the exact rate of convergence of
the constant $d_n$ for $p=2$
\[
d_n=4-\frac{16\pi^2}{\ln^2n}+O\left(\frac{\ln\ln n}{\ln^3n} \right).
\]

In \cite{FS} F. Stampach gave slightly better estimation, i.e.
\[
d_n=4-\frac{16\pi^2}{\ln^2n}+\frac{32\pi^2(\gamma+6\log2)}{\log^3n}+
O\left(\frac{1}{\ln^4n} \right).
\]

In \cite{DIGR} we also studied  the asymptotic behavior of the constant $d_n$ for $p=2$. It was proved there that $d_n$ can be expressed in terms of the smallest zero of a continuous dual Hahn polynomial of degree $n$ (see \cite{Long}), for a specific choice of the parameters, in terms of which these polynomials are defined. Despite that nice interpretation of $d_n$, it was only proved in \cite[Theorem 1.1]{DIGR} that  the next inequalities are true for every natural $n\geq 3$
\begin{equation}\label{eq022}
4\Bigg(1-\frac{4}{\ln n +4}\Bigg)\leq d_n \leq
4\Bigg(1-\frac{8}{(\ln n + 4)^2}\Bigg).
\end{equation}

In all proofs of the above mentioned estimations for the constant $d_n$, the authors substantially used the special properties of the space $l_2$. 
In \cite {DGM} and \cite {GG} we applied a different approach 
which allowed us to give a simpler proof of some of
the mentioned estimations and to find an almost extremal sequence. We proved the next theorem.
 
\begin{theorem}\cite{DGM}\label{th02}
Let
\begin{equation}\label{eq028}
a_k=\int_k^{k+1}h(x)dx,
\end{equation}
where 
\begin{equation*}\label{eq020}
h(x)=x^{-1/2}\left(2\alpha\cos(\alpha\ln x)+\sin(\alpha\ln x) \right),\quad 1\le x\le n+1,
\end{equation*}
and $\alpha$ is the only solution of the equation
\[
\tan(\alpha\ln (n+1))+2\alpha=0\ \ \mathrm{in\ the\ interval}\ \ \left(\frac{\pi}{2\ln(n+1)},\frac{\pi}{\ln(n+1)} \right).
\]
 Then 
\begin{equation}\label{eq018}
\sum_{k=1}^{n}\Big(\frac{1}{k}\sum_{j=1}^{k}a_j\Big)^2\geq
\frac{4}{1+4\alpha^2}\,\sum_{k=1}^{n}{a_k^2}.
\end{equation}
 \end{theorem}

By combining the results \eqref{eq022} and \eqref{eq018} the exact rate of convergence of $\{ d_n \}$ is established and
the next very sharp estimates for $d_n$, i.e. the next inequalities 
\begin{equation*}\label{estdn}
4-\frac{16\pi^2}{\ln^2(n+1)} \le d_n \le 
4-\frac{32}{(\ln n + 4)^2}
\end{equation*}
hold for every natural $n\geq 3$. Also the sequence $\{a_k\}_1^n$ defined in \eqref{eq028} is the almost
extremal sequence.

 In this paper we establish very sharp estimates for $d(a,b)$ and $d_n$ for $2\le p<\infty$, as well as obtain an ``almost extremal'' function for \eqref{eqI15} 
and ``almost extremal'' sequence for \eqref{eqI16}. Our main results are summarized in the following two theorems.  

\begin{theorem}\label{th1}
Let $2\le p<\infty$ and $0<a<b<\infty$. Then there exist positive constants $c_1=c_1(p)$ and  $c_2=c_2(p)$, depending only on $p$, 
  such that the next estimates
for the constant $d(a,b)$ in \eqref{eqI15} hold 
\begin{equation}\label{MR1}
\left(\frac{p}{p-1} \right)^p  \left(1-\frac{c_1}{\ln^2 \frac{b}{a}}\right)\le
d(a,b)\,\le\left(\frac{p}{p-1} \right)^p  \left(1+\frac{c_2}{\ln^2 \frac{b}{a}}\right)^{-1}.
\end{equation}
Moreover, the function 
\[
f^\ast(x)=\frac{1}{x^{1/p}}\left(\frac{\alpha p}{p-1} \cos(\alpha\ln x)+\sin(\alpha\ln x)\right)
\]
where $\alpha$ is the only solution of the equation
\[
\tan\left(\alpha\ln \frac{b}{a} \right)+\frac{\alpha p}{p-1}=0
\]
in the interval $\left(\frac{\pi}{2\ln \frac{b}{a}},\frac{\pi}{\ln \frac{b}{a}} \right)$,
is an "almost extremal" function in the sense that 
\begin{equation*}
\int_a^b\left(\frac{1}{x}\int_a^x f^\ast(t)dt\right)^p\,dx\geq
\left(\frac{p}{p-1} \right)^p  \left(1-\frac{c_1}{\ln^2 \frac{b}{a}}\right)\,\int_a^b [f^\ast(x)]^p dx.
\end{equation*}
\end{theorem}

An immediate consequence of Theorem \ref{th1} is the next corollary. 
\begin{cor}\label{cor1}
 When either of the limits relations $a\rightarrow0$,  $b\rightarrow\infty$, or both hold, i.e. $\ln(b/a)\rightarrow\infty$, then
\[
d(a,b) \sim 
\left(\frac{p}{p-1} \right)^p -\frac{C}{\ln^2 \frac{b}{a}}.
\]
More precisely there exist  constants $C_1(p)>0$ and $C_2(p)>0$ depending only on $p$ such that
\begin{equation*}
\left(\frac{p}{p-1} \right)^p -\frac{C_1}{\ln^2 \frac{b}{a}} \le d(a,b) \le 
\left(\frac{p}{p-1} \right)^p -\frac{C_2}{\ln^2 \frac{b}{a}}.
\end{equation*}
Also,  for the constant $C_1$ 
the next estimation holds
\[
C_1\le \left(\frac{p}{p-1} \right)^{p+1}p\pi^2
\]
which is the exact constant for $p=2$.
\end{cor}
\begin{remark}
Because of the close connection to the maximal function also it is natural to consider the inequality
\begin{equation*}
\int_a^b\left(\frac{1}{x-a}\int_a^xf(t)dt\right)^p\,dx\leq
d(a,b)\,\int_a^b [f(x)]^p dx, \qquad f(x)\geq0,\ \ f \in L^p[a,b]
\end{equation*}
which is equivalent (by change of variables) to
\begin{equation*}
\int_0^b\left(\frac{1}{x}\int_0^xf(t)dt\right)^p\,dx\leq
d(b)\,\int_0^b [f(x)]^p dx, \qquad f(x)\geq0,\ \ f \in L^p[0,b].
\end{equation*}
Then from the above corollary we obtain that $d(b)=\left(\frac{p}{p-1} \right)^p$ and the only function for which the equality is attained is $f=0$ a.e.
\end{remark}

\begin{theorem}\label{th2} 
Let $2\le p<\infty$. Then there exist positive constants $c_3=c_3(p)$ and  $c_4=c_4(p)$, depending only on $p$, 
  such that for every natural $n\geq 2$ the next estimates
for the constant $d_n$ in \eqref{eqI16} hold 
\begin{equation}\label{MR2}
\left(\frac{p}{p-1} \right)^p-\frac{c_3}{\ln^2 n}\le
d_n\,\le\left(\frac{p}{p-1} \right)^p-\frac{c_4}{\ln^2 n}.
\end{equation}
Moreover, the sequence
\[
a_k^\ast = \int_k^{k+1}f^*(x)\, dx,\quad k=1,2,...,n,
\] 
where 
\begin{equation}\label{eq20}
f^*(x)=\frac{1}{x^{1/p}}\left(\frac{\alpha p}{p-1}\cos(\alpha\ln x)+\sin(\alpha\ln x)\right),\quad 1\le x\le n+1
\end{equation}
  and $\alpha$ is the only solution of the equation
\[
\tan(\alpha\ln (n+1))+\frac{\alpha p}{p-1}=0
\]
in the interval $\left(\frac{\pi}{2\ln(n+1)},\frac{\pi}{\ln(n+1)} \right)$, is an "almost extremal" sequence in the sense that 
\begin{equation}\label{eqR1}
\sum_{k=1}^{n}\left(\frac{1}{k}\sum_{j=1}^{k}a_j^\ast\right)^p\geq
\left(\left(\frac{p}{p-1} \right)^p-\frac{c_3}{\ln^2 n}\right)\,\sum_{k=1}^{n}{[a_k^\ast]^p}.
\end{equation}
Also, leting $n\rightarrow\infty$ for the constant $c_3$  the next estimation holds
\[
c_3\le \left(\frac{p}{p-1} \right)^{p+1}p\pi^2
\]
which is the exact constant for $p=2$.
\end{theorem}


\begin{remark}
The constants  $c_2(p)$ and $c_4(p)$ in Theorem \ref{th1} and Theorem \ref{th2} are by no means the
best ones. They could be improved in a lot of ways but that could have made the proofs longer and more complicated.
Our goal was to establish the exact rate of convergence and to keep the proofs as simple as possible. 

Henceforth, the constant $c=c(p)$ will always be an absolute constant, which means it will depend only on $p$. Also, it may be different on each occurrence. The relation $O(f)$ means that there exists a positive constant $c(p)$, depending only on $p$ such that $\left|O(f)\right|\le c(p)|f|$.
\end{remark}

The paper is organized as follows. In Section 2, some technical results are proved, in  Section 3, the
right inequality of Theorem \ref{th1} is proved, in  Section 4, the left inequality of Theorem \ref{th1} is proved, in  Section 5, the left inequality of Theorem \ref{th2} is proved and in  Section 6, the
right inequality of Theorem \ref{th2} is proved.

\section{Auxiliary Results}
We need some technical lemmas.

\begin{lemma} \label{lem01}
For $0\le x\le 1$ and $\alpha\geq 0$ the next inequality is true:
\begin{equation}\label{eq1}
(1-x)^\alpha \le 1-\alpha x+\frac{1}{2}(\alpha x)^2.		
\end{equation}
	\end{lemma}
	\begin{proof}
	From	$1-x \le e^{-x}$
	we have for $\alpha\geq 0$ and $\alpha x<1$
	\[
	(1-x)^{\alpha}\le e^{-\alpha x}=1-\alpha x+\frac{(\alpha x)^2}{2}+\sum_{k=3}^{\infty}(-1)^k\frac{(\alpha x)^k}{k!}
	\le 1-\alpha x+\frac{1}{2}(\alpha x)^2.	
	\]
	For $\alpha x\geq1$ the function $f(x)=1-\alpha x+\frac{1}{2}(\alpha x)^2-(1-x)^\alpha$ is increasing and 
	consequently $f(x)\geq f(0)=0$.	
	\end{proof}
	
\begin{lemma} \label{lem02}
For $x\geq -1$, $0\le \alpha\le 1$  the next inequality is true:
\begin{equation}\label{eq2}
(1+x)^\alpha \le 1+\alpha x.
\end{equation}
	\end{lemma}
	\begin{proof}
	Let $\alpha =1/\beta$. Then \eqref{eq2} follows from the Bernoulli inequality.
	\end{proof}
	
	\begin{lemma} \label{lem021}
For $x\geq 0$, $\alpha\geq 0$ and $\alpha x<1$  the next inequality is true:
\begin{equation}\label{eq234}
(1+x)^\alpha \le 1+\alpha x+(\alpha x)^2.
\end{equation}
	\end{lemma}
	\begin{proof}
	\[
	(1+x)^\alpha< e^{\alpha x}
	=1+\alpha x+\sum_{k=2}^{\infty}\frac{(\alpha x)^k}{k!}
	\le 1+\alpha x+(\alpha x)^2\sum_{k=2}^{\infty}\frac{1}{k!}
	\le 1+\alpha x+(\alpha x)^2.
	\]
	\end{proof}
		
	\begin{lemma} \label{lem21}
	Let $b>1$, $p\geq 2$, $\frac{1}{p}+\frac{1}{q}=1$ and
\begin{equation}\label{eq200}
g(x)=\frac{1}{x^{1/(pq)}}\left(\cos(\alpha\ln x) \right)^{1/q}
\end{equation}
where $\alpha=\frac{1}{\ln b}\arctan\frac{1}{p}$ .
	Then there exists a constant $c=c(p)>0$, depending only on $p$ such that for $1\le x\le b$
the next inequality holds:
\begin{equation}\label{eqI3}
\left(\frac{\cos(\alpha\ln x)+\alpha q\sin(\alpha\ln x)}{1+\alpha^2q^2} \right)^{p/q}\le
-q\left(1+c\alpha^2 \right)^{-1}x^{1+1/q}\left[(g(x))^p \right]'.
\end{equation}
	\end{lemma}
\begin{proof}
The above inequality \eqref{eqI3} is equivalent to
\begin{align*}
&\left(\frac{\cos(\alpha\ln x)+\alpha q\sin(\alpha\ln x)}{1+\alpha^2q^2} \right)^{p/q}\\
&\le \left(1+c\alpha^2 \right)^{-1}\left(\cos(\alpha\ln x)\right)^{p/q-1}  (\cos(\alpha\ln x)+\alpha p\sin(\alpha\ln x))
\end{align*}
or
\[
\left(\frac{1+\alpha qy}{1+\alpha^2q^2} \right)^{p/q}\le 
\frac{1+\alpha py}{1+c\alpha^2}
\]
where $y=\tan(\alpha\ln x)$. Obviously $y=\tan(\alpha\ln x)\le\tan(\alpha\ln b)=1/p$.
We consider two cases.

\textbf{Case 1.} $\alpha\geq1.$\\
Since
\[
\frac{1+\alpha qy}{1+\alpha^2q^2}\le \frac{1+\alpha q/p}{1+\alpha^2}
\le \frac{1+\alpha }{1+\alpha^2}\le1
\]
we have
\[
\left(\frac{1+\alpha qy}{1+\alpha^2q^2} \right)^{p/q}
=\frac{1+\alpha qy}{1+\alpha^2q^2} \left(\frac{1+\alpha qy}{1+\alpha^2q^2} \right)^{p/q-1}
\le \frac{1+\alpha py}{1+\alpha^2q^2}.
\]

\textbf{Case 2.} $\alpha<1.$\\
 From lemma \ref{lem021}
\[
(1+\alpha qy)^{p/q}\le 1+\alpha py+(\alpha py)^2
\] 
and from Bernoulli's inequality
\[
\left(1+\alpha^2q^2 \right)^{p/q}\geq 1+\alpha^2 pq.
\]
So, it is enough to prove that there exists a constant $c=c(p)>0$ such that
\[
 1+\alpha py+(\alpha py)^2\le
\frac{1+\alpha^2 pq}{1+c\alpha^2}(1+\alpha py)
\]
which after some simplifications is
\[
(py)^2\le \frac{pq-c}{1+c\alpha^2}. 
\]
Since $py\le 1$ and $\alpha<1$ the above inequality is true if we take, for instance, $c=(pq-1)/2$.

By taking $c=\min\{q^2, (pq-1)/2 \}$ we complete the proof of the lemma.
\end{proof}
\begin{remark}
The inequality \eqref{eqI3} could be written in the following way
\begin{equation}\label{eqI31}
\left(\frac{\cos(\alpha\ln x)+\alpha q\sin(\alpha\ln x)}{1+\alpha^2q^2} \right)^{p/q}\le
-q\left(1+\frac{c}{\ln^2b} \right)^{-1}x^{1+1/q}\left(g^{p}(x) \right)'
\end{equation}
where $c=\arctan (1/p)\min\{q^2, (pq-1)/2 \}$.
\end{remark}

	\begin{lemma} \label{lem20}
	Let  $p\geq 2$, $\frac{1}{p}+\frac{1}{q}=1$, $0<\epsilon<1$, 
   \[
   b_0=e^{\pi/\sqrt{\min\{q(p-q)^{-1}(pq+1)^{-2},4(pq)^{-2} \}\epsilon}}
  \]
and
\begin{equation}\label{eq21}
f^*(x)=\frac{1}{x^{1/p}}(\alpha q\cos(\alpha\ln x)+\sin(\alpha\ln x))
\end{equation}
where $\alpha$ is the only solution of the equation
\[
\tan\left(\alpha\ln b \right)+\alpha q=0
\]
in the interval $\left(\frac{\pi}{2\ln b},\frac{\pi}{\ln b} \right)$.
Then for every $b>b_0$ the next inequality holds
\begin{align}\label{eqI5}
(\sin(\alpha\ln x))^{p/q}\geq 
-\frac{qx^{1+1/q}\left[(f^*(x))^{p/q} \right]'} {1+(pq+\epsilon)\alpha^2 }
\end{align}
	\end{lemma}
	\begin{remark}
	The function $f^*$ defined by \eqref{eq21} is well defined. Indeed, if 
	$\alpha\ln x\in \left(0,\frac{\pi}{2}\right]$ it is obvious. If 
	$\alpha\ln x\in \left(\frac{\pi}{2},\pi \right)$ since the function
	$h(x)=\alpha q\cos(\alpha\ln x)+\sin(\alpha\ln x)$ 
	is decreasing and $h(b)=0$ it follows that $h(x)>h(b)=0$, i.e. $h(x)>0$ for $1\le x\le b$ and
	 consequently $f^*(x)>0$.
	\end{remark}

\begin{proof}
It is easy to see that for every $b>b_0$
  \begin{equation*}\label{eqalpha}
  \alpha \le \sqrt{\min\left\{\frac{q}{(p-q)(pq+1)^2}, \frac{4}{(pq)^2}\right \}\epsilon}.
  \end{equation*}
Since
\begin{align}\label{eqI4}
&-\frac{qx^{1+1/q}}{1+\alpha^2 pq}\left[(f^*(x))^{p/q} \right]'\notag\\
&=(\alpha q\cos(\alpha\ln x)+\sin(\alpha\ln x))^{p/q-1}
\left(\sin(\alpha\ln x)-\frac{\alpha (p-q)}{1+\alpha^2 pq}\cos(\alpha\ln x) \right)
\end{align}
 we need to prove that for every $b>b_0$  the next inequality is true
\begin{align}\label{eqI6}
&\frac{1+(pq+\epsilon)\alpha^2 }{1+pq\alpha^2}(\sin(\alpha\ln x))^{p/q}\notag\\
&\geq	(\alpha q\cos(\alpha\ln x)+\sin(\alpha\ln x))^{p/q-1}
	\left(\sin(\alpha\ln x)-\frac{\alpha (p-q)}{1+\alpha^2 pq}\cos(\alpha\ln x) \right).
\end{align}
	We consider two cases.
	
	\textbf{Case 1.} $\alpha\ln x\in \left[\frac{\pi}{2},\pi \right)$.\\	
	Let $\alpha\ln x=\pi-\phi,\, 0<\phi\le\pi/2$. Then\eqref{eqI6} is equivalent to 
	\begin{align*}
\frac{1+(pq+\epsilon)\alpha^2 }{1+pq\alpha^2}(\sin(\alpha\ln x))^{p/q}
\geq (-\alpha q\cos\phi+\sin\phi)^{p/q-1}
	\left(\sin\phi+\frac{\alpha (p-q)}{1+\alpha^2 pq}\cos\phi \right)
\end{align*}
	or
	\begin{align}\label{I61}
\left[1+(pq+\epsilon)\alpha^2 \right]y^{p/q}\geq \left(y-\alpha q \right)^{p/q-1}
\left[\left(1+\alpha^2pq \right)y+\alpha (p-q) \right]
\end{align}
where $y=\tan\phi$. Since $y>\alpha q$ in this case, we have from Bernoulli's inequality
\[
\left(\frac{y}{y-\alpha q} \right)^{p/q}\geq \frac{y+\alpha(p-q)}{y-\alpha q}.
\]
So, it is enough to prove that
\[
\left[1+(pq+\epsilon)\alpha^2 \right][y+\alpha(p-q)]\geq \left(1+\alpha^2 pq \right)y+\alpha (p-q)
\]
which is easy to verify.

	\textbf{Case 2.} $\alpha\ln x\in \left[0,\frac{\pi}{2} \right)$.\\
		In this case \eqref{eqI6} is equivalent to
	\begin{align}\label{eqI9}
\left(1+(pq+\epsilon)\alpha^2\right )y^{p/q}
\geq (y+\alpha q )^{p/q-1}
	\left[\left(1+\alpha^2 pq\right )y-\alpha(p-q)  \right]
\end{align}
	where $y=\tan(\alpha\ln x)$.
	If $	y\le \frac{\alpha (p-q)}{1+\alpha^2 pq}$ then
	\eqref{eqI9} is obvious.
	Let
	$y> \frac{\alpha (p-q)}{1+\alpha^2 pq}$.

	\textbf{Case 2.1.} $2\le p< 3$\\
	\eqref{eqI9} is equivalent to
	\begin{align*}
1+(pq+\epsilon)\alpha^2\geq \left(1+\frac{\alpha q}{y} \right)^{p/q-1} 
\left(1+\alpha^2 pq-\frac{\alpha(p-q)}{y} \right).
\end{align*}	
	Since
	\[
	\left(1+\frac{\alpha q}{y} \right)^{p/q-1} \le 1+\frac{\alpha(p-q)}{y}
	\]
	it is enough to prove that 
	\begin{align*}
1+(pq+\epsilon)\alpha^2\geq \left(1+\frac{\alpha(p-q)}{y}\right)
\left(1+\alpha^2pq-\frac{\alpha(p-q)}{y} \right).
\end{align*}
	Simplifying it
	\begin{align*}
\epsilon
\geq \frac{\alpha pq(p-q)}{y}-\frac{(p-q)^2}{y^2}
	\end{align*}
	which is true since $\alpha pq<2\sqrt\epsilon$.

\textbf{Case 2.2.} $p\geq 3$\\
	\eqref{eqI9} is equivalent to
	\begin{align*}
\left[1+(pq+\epsilon)\alpha^2\right ]\left(\frac{y}{y+\alpha q}\right)^{p/q-1}
\geq 1+\alpha^2pq-\frac{\alpha(p-q)}{y} .
\end{align*}	
	From Bernoulli's inequality
	\begin{align*}
\left(\frac{y}{y+\alpha q}\right)^{p/q-1}
\geq 1-\frac{ \alpha(p-q)}{y+\alpha q }. 
\end{align*}
	So, it is enough to prove that 
	\begin{align*}
  \left[1+(pq+\epsilon)\alpha^2\right ]\left( 1-\frac{\alpha (p-q)}{y+\alpha q} \right)
	\geq 1-\frac{\alpha (p-q)}{(1+\alpha^2 pq)y}.
\end{align*}
Simplifying it
\[
\epsilon+\frac{q(p-q)}{y(y+\alpha q)}\geq \frac{(p-q)(pq+\epsilon)\alpha}{y+\alpha q}.
\]
If $y\le q[(pq+\epsilon)\alpha]^{-1}$ it is obvious. For $y> q[(pq+\epsilon)\alpha]^{-1}$ 
	\begin{align*}
  \frac{(p-q)(pq+\epsilon)\alpha}{y+\alpha q}<\frac{(p-q)(pq+\epsilon)\alpha}
{q[(pq+\epsilon)\alpha]^{-1}+\alpha q}<\epsilon
         \end{align*}
	since 
\[
 \alpha^2<\frac{q\epsilon}{(p-q)(pq+1)^2}.
\]

The lemma is proved.
\end{proof}
	
	\begin{lemma} \label{lem05}
For $p\geq 2$, $\frac{1}{p}+\frac{1}{q}=1$ and every natural numbers $i$ and $n$ such that
$i\le n$ the next inequality is true:
\begin{equation}\label{eq5}
\sum_{k=i}^{n}\frac{1}{k^{1+1/q}}\left(1-\frac{1}{pk^{1/q}}\right)^{p/q}
\le q\left(\frac{1}{i^{1/q}}-\frac{1}{(n+1)^{1/q}} \right).	
\end{equation}
	\end{lemma}
	
	\begin{proof}
		We will prove that there is a $k_0 \in \mathbb{N}$ such that for every $k\geq k_0$ the next inequality is true:
	\begin{equation}\label{eq51}
	\frac{1}{k^{1+1/q}}\left(1-\frac{1}{pk^{1/q}}\right)^{p/q}
\le q\left(\frac{1}{k^{1/q}}-\frac{1}{(k+1)^{1/q}} \right).	
	\end{equation}
	From \eqref{eq1}
	\[
	\left(1-\frac{1}{pk^{1/q}}\right)^{p/q}\le 1-\frac{1}{qk^{1/q}}+\frac{1}{2q^2k^{2/q}}
	\]
	so, it is enough to prove that   there is a $k_0 \in \mathbb{N}$ such that for every $k\geq k_0$
	\begin{equation*}\label{eq52}
	\frac{1}{k^{1+1/q}}\left(1-\frac{1}{qk^{1/q}}+\frac{1}{2q^2k^{2/q}} \right)
	\le q\left(\frac{1}{k^{1/q}}-\frac{1}{(k+1)^{1/q}} \right)
	\end{equation*}
	i.e.
	\begin{equation*}\label{eq53}
	1-\frac{1}{qk^{1/q}}+\frac{1}{2q^2k^{2/q}} 
	\le qk\left[1-\left(\frac{k}{k+1} \right)^{1/q} \right].
	\end{equation*}
	But from \eqref{eq2} 
	\[
	\left(\frac{k}{k+1} \right)^{1/q}=\left(1-\frac{1}{k+1} \right)^{1/q}
	\le 1-\frac{1}{q(k+1)}.
	\]
	and then
	\[
	qk\left[1-\left(\frac{k}{k+1} \right)^{1/q} \right]\geq \frac{k}{k+1}=1-\frac{1}{k+1}.
	\]
	So, it is enough to prove that  there is a $k_0 \in \mathbb{N}$ such that for every 
$k\geq k_0$
	\begin{equation*}\label{eq54}
	1-\frac{1}{qk^{1/q}}+\frac{1}{2q^2k^{2/q}}\le  1-\frac{1}{k+1},
	\end{equation*}
i.e.
      \[
    \frac{1}{k+1}+\frac{1}{2q^2k^{2/q}} \le \frac{1}{qk^{1/q}}
      \]
	which is obviously true for $k$ big enough. Actually, by considering the function
	$f(x)=\frac{1}{xk^{1/x}}+\frac{1}{2x^2k^{2/x}}$ it is not difficult to prove that
	it is true for every $k\geq 8$.
	From the above it follows that there is $i_0$ such that for every $i\geq i_0$ 
	\eqref{eq51} is true and consequently \eqref{eq5} as well.
	
	Now, let \eqref{eq5} is true for $i+1$. We will prove that it is true for $i$ as well.
	We have
	\[
	\sum_{k=i}^{n}\frac{1}{k^{1+1/q}}\left(1-\frac{1}{pk^{1/q}}\right)^{p/q}
\le \frac{1}{i^{1+1/q}}\left(1-\frac{1}{pi^{1/q}}\right)^{p/q}
+q\left(\frac{1}{(i+1)^{1/q}}-\frac{1}{(n+1)^{1/q}} \right)	
	\]
	So, it is enough to prove that
	\[
	\frac{1}{i^{1+1/q}}\left(1-\frac{1}{pi^{1/q}}\right)^{p/q}+
	\frac{q}{(i+1)^{1/q}}\le\frac{q}{i^{1/q}}
	\]
	i.e.
	\[
	\frac{1}{qi}\left(1-\frac{1}{pi^{1/q}}\right)^{p/q}+
	\left(\frac{i}{i+1} \right)^{1/q}\le 1.
	\]
	 and from \eqref{eq2} it follows that it is  enough to prove that
	\begin{equation*}\label{eq55}
	\left(1-\frac{1}{pi^{1/q}}\right)^{p/q}\le \frac{i}{i+1}.
	\end{equation*}
	From Bernoulli's inequality for $\alpha\geq 1$ and $0\le x<1$
	\[
	(1-x)^\alpha\le 1-\frac{\alpha x}{1+(\alpha-1)x}
	\]
	we have
	\[
	\left(1-\frac{1}{pi^{1/q}}\right)^{p/q}\le 1-\frac{p}{pqi^{1/q}+p-q}.
	\]
	So, it is enough to prove
	\[
	\frac{1}{i+1}\le \frac{p}{pqi^{1/q}+p-q}
	\]
	i.e.
	\[
	i^{1/q}\le \frac{i}{q}+\frac{1}{p}
	\]
	which is easy to verify.
			
	The proof of the lemma is complete.
		\end{proof}

	\begin{lemma} \label{lem06}
For $p\geq 2$, $\frac{1}{p}+\frac{1}{q}=1$ and every natural numbers $i$ and $n$ such that
$i\le n$ the next inequality is true:
\begin{align}\label{eq6}
\sum_{k=i}^{n}&\frac{\ln^2k-2q\ln k+2q^2}{k^{1+1/q}}\notag\\
&>\frac{q\left(\ln^2i+2q^2 \right)}{i^{1/q}}
+\frac{\ln^2i-2q\ln i+2q^2}{2i^{1+1/q}}-\frac{q\left(\ln^2n+2q^2 \right)}{n^{1/q}}.
\end{align}
	\end{lemma}
	\begin{proof}
	For the function
	\[
	f(x)=\frac{\ln^2x-2q\ln x+2q^2}{x^{1+1/q}}
	\]
	we have
	\[
	f'(x)=\frac{1}{x^{2+1/q}}\left[-\left(1+\frac{1}{q}\right)\ln^2 x+2(q+2)\ln x-2q(q+2) \right]
	\]
	and
	\[
	f''(x)=\frac{1}{x^{3+1/q}}
	\left[\left(2+\frac{3}{q}+\frac{1}{q^2}\right)\ln^2 x-2(6+2q+\frac{3}{q})\ln x+4\left(q^2+3q+2\right) \right].
	\]
	It is easy to see that $f'(x)<0$, $f''(x)>0$ and consequently
$f'(x)$ is increasing and $|f'(x)|$ is decreasing.	
 From Euler's summation formula (see, for instance \cite [p.149]{Apostol})
\[
\sum_{k=i}^nf(k)=\int_i^nf(x)dx+\frac{f(i)+f(n)}{2}+\int_i^n\left(x-[x]-\frac{1}{2} \right)f'(x)dx
\]
where $[x]$ is the floor function.

\[
\int_i^n\left(x-[x]-\frac{1}{2} \right)f'(x)dx=\sum_{k=i}^{n-1}\left(\int_k^{k+1/2}+\int_{k+1/2}^{k+1} \right)>0
\]
since
\[
\int_k^{k+1/2}>0,\,\,\int_{k+1/2}^{k+1}<0\quad\mbox{and}\quad \int_k^{k+1/2}>\left|\int_{k+1/2}^{k+1} \right|.
\]
Then
\begin{align*}
\sum_{k=i}^nf(k)&>\int_i^nf(x)dx+\frac{f(i)+f(n)}{2}
>\int_i^nf(x)dx+\frac{\ln^2i-2q\ln i+2q^2}{2i^{1+1/q}}\\
&=\frac{q\left(\ln^2i+2q^2 \right)}{i^{1/q}}
+\frac{\ln^2i-2q\ln i+2q^2}{2i^{1+1/q}}-\frac{q\left(\ln^2n+2q^2 \right)}{n^{1/q}}.
\end{align*}
The lemma is proved.
	\end{proof}
		
\begin{lemma} \label{lem07}
For $p\geq 2$, $\frac{1}{p}+\frac{1}{q}=1$ and every natural numbers $i$ and $n$ such that
$i\le n$ the next inequality is true:
\begin{align}\label{eq7}
\sum_{k=i}^{n}&\frac{\ln^2k-2q\ln k+2q^2}{k^{1+2/q}}
<\frac{\ln^2i-2q\ln i+2q^2}{i^{1+2/q}}
+\frac{q\left(\ln^2i-q\ln i+3/2q^2\right)}{2i^{2/q}}.
\end{align}
	\end{lemma}	
		\begin{proof}
	For the function
	\[
	g(x)=\frac{\ln^2x-2q\ln x+2q^2}{x^{1+2/q}}
	\]
	we have
	\[
	g'(x)=\frac{1}{x^{2+2/q}}\left[-\left(1+\frac{2}{q}\right)\ln^2 x+2(q+3)\ln x-2q(q+3) \right]<0
	\]
	and consequently $g(x)$ is decreasing and
	\begin{align*}
\sum_{k=i}^ng(k)&
<\frac{\ln^2i-2q\ln i+2q^2}{i^{1+2/q}}+\int_i^ng(x)dx\\
&<\frac{\ln^2i-2q\ln i+2q^2}{i^{1+2/q}}
+\frac{q\left(\ln^2i-q\ln i+3/2q^2\right)}{2i^{2/q}}.
\end{align*}
The lemma is proved.
	\end{proof}

\begin{lemma} \label{lem09}
For $p\geq 2$, $\frac{1}{p}+\frac{1}{q}=1$ and every natural numbers $i$ and $n$ such that
$i\le n$  the next inequality is true:
\begin{align}\label{eq9}
\sum_{k=i}^n\frac{1}{k^p} &\left[\left(k^{1/q} \right)^{p/q-1}-\left(k^{1/q}-\frac{1}{p} \right)^{p/q-1} \right]
\left[k^{1/q}\left(\ln^2k-2q\ln k+2q^2 \right)-2q^2 \right]\notag\\
&<\frac{\ln^2i-2q\ln i+2q^2}{qi^{1+2/q}}+\frac{\ln^2i-q\ln i+3/2q^2}{2i^{2/q}}
-\frac{2q^2}{3i^{3/q}}+\frac{2q^2}{3n^{3/q}}.
\end{align}
	\end{lemma}	
	\begin{proof}
	Since
	\[
	\left(1-\frac{1}{pk^{1/q}}\right)^{p/q-1}
	>\left(1-\frac{1}{pk^{1/q}}\right)^{p/q}
	>1-\frac{1}{qk^{1/q}}
	\]	
	we have
	\begin{align*}
\left(k^{1/q} \right)^{p/q-1}-\left(k^{1/q}-\frac{1}{p} \right)^{p/q-1}
=\left(k^{1/q} \right)^{p/q-1}\left[1-\left(1-\frac{1}{pk^{1/q}} \right)^{p/q-1} \right]
<\frac{1}{q}\left(k^{1/q} \right)^{p/q-2}.
\end{align*}
Then
\begin{align*}
\sum_{k=i}^n\frac{1}{k^p}& \left[\left(k^{1/q} \right)^{p/q-1}-\left(k^{1/q}-\frac{1}{p} \right)^{p/q-1} \right]
\left[k^{1/q}\left(\ln^2k-2q\ln k+2q^2 \right)-2q^2 \right]\notag\\
&<\frac{1}{q}\sum_{k=i}^n\frac{1}{k^{1+3/q}}\left[k^{1/q}\left(\ln^2k-2q\ln k+2q^2 \right)-2q^2 \right]\notag\\
&=\frac{1}{q}\sum_{k=i}^n\frac{\ln^2k-2q\ln k+2q^2}{k^{1+2/q}}
-2q\sum_{k=i}^n\frac{1}{k^{1+3/q}}\\
&<\frac{1}{q}\sum_{k=i}^n\frac{\ln^2k-2q\ln k+2q^2}{k^{1+2/q}}
-2q\int_{i}^n\frac{dx}{x^{1+3/q}}\\
&=\frac{1}{q}\sum_{k=i}^n\frac{\ln^2k-2q\ln k+2q^2}{k^{1+2/q}}
-\frac{2q^2}{3i^{3/q}}+\frac{2q^2}{3n^{3/q}}
\end{align*}
and the lemma follows from \eqref{eq7}.
	\end{proof}
		
	\begin{lemma} \label{lem10}
For $p\geq 2$, $\frac{1}{p}+\frac{1}{q}=1$ and every natural $i$  there is a constant $c=c(p)>0$,
 depending only on $p$ such that for every natural $i\geq2$  the next inequality is true:
\begin{align*}\label{eq10}
2q^2-\frac{3}{4}+\frac{2q}{3i^{2/q}}-\frac{2q}{i^{1+1/q}}-\frac{q^2}{i^{1/q}}
-\frac{\ln^2i-q\ln i+3/2q^2}{2qi^{1/q}}>c(p).
\end{align*}
	\end{lemma}	
	\begin{proof}
	It is obvious that there is a $i_0$ and a constant $c(p,i_0)$ such that for every $i>i_0$ the above inequality is true.	For  $i\le i_0$	we have
	\begin{align*}
&2q^2-\frac{3}{4}+\frac{2q}{3i^{2/q}}-\frac{2q}{i^{1+1/q}}-\frac{q^2}{i^{1/q}}
-\frac{\ln^2i-q\ln i+3/2q^2}{2qi^{1/q}}\\
&>2q^2-\frac{3}{4}-\frac{4q}{3i^{1+1/q}}-\frac{q^2}{i^{1/q}}
-\frac{\left(\ln i-q/2\right)^2+5/4q^2}{2qi^{1/q}}\\
&=\frac{1}{i^{1/q}}\left[ \left(2q^2-\frac{3}{4}\right)i^{1/q}-\frac{4q}{3i}-q^2
-\frac{\left(\ln i-q/2\right)^2+5/4q^2}{2q}\right]
\end{align*}
	Now we will prove that for every $i\le i_0$
	\begin{align*}
\left(2q^2-\frac{3}{4}\right)i^{1/q}>q^2+\frac{4q}{3i}
+\frac{\left(\ln i-q/2\right)^2+5/4q^2}{2q}.
\end{align*}
	We consider the cases $i=2$ and $i\geq3$ separately.
	
\textbf{Case 1.} $i=2$\\	
	We need to prove that 
		\begin{align*}
\left(2q^2-\frac{3}{4}\right)2^{1/q}>q^2+\frac{2}{3}q
+\frac{\left(\ln 2-q/2\right)^2+5/4q^2}{2q}.
\end{align*}
	Since $\left(\ln 2-q/2\right)^2<0.1$ it is enough to prove
	\begin{align*}
\left(2q^2-\frac{3}{4}\right)2^{1/q}>q^2+\frac{2}{3}q+\frac{1}{20q}+\frac{5}{8}q
=q^2+\frac{31}{24}q+\frac{1}{20q}.
\end{align*}
	Considering for $1\le x\le2$ the function
	\[
	f(x)=\left(2x^2-\frac{3}{4}\right)2^{1/x}-x^2-\frac{31}{24}x-\frac{1}{20x}
	\]
	we have
	\begin{align*}
	f'(x)&=4x2^{1/x}-\left(2-\frac{3}{4x^2} \right)2^{1/x}\ln 2-2x-\frac{31}{24}+\frac{1}{20x^2}\\
	&>4x2^{1/x}-\frac{7}{10}\left(2-\frac{3}{4x^2} \right)2^{1/x}-2x-\frac{31}{24}\\
	&=\left(4x-\frac{7}{5}+\frac{21}{40x^2} \right)2^{1/x}-2x-\frac{31}{24}\\
	&>4\sqrt2x-\frac{7\sqrt2}{5}-2x-\frac{31}{24}> 4\sqrt2-\frac{7\sqrt2}{5}-2-\frac{31}{24}>0
	\end{align*}
	and consequently the function $f(x)$ is increasing and $f(x)>f(1)>0$.
	
	\textbf{Case 2.} $i\geq3$\\
		We have
	\[
	i^{1/q}=e^{\ln i/q}>1+\frac{\ln i}{q}+\frac{\ln^2 i}{2q^2}\quad\mbox{and}\quad
	\left(\ln i-q/2\right)^2<\ln^2i
	\]
	so it is enough to prove that for $3\le i$ the next inequality is true
	\[
	\left(2q^2-\frac{3}{4} \right)\left(1+\frac{\ln i}{q}+\frac{\ln^2 i}{2q^2} \right)
	>q^2+\frac{4q}{3i}+\frac{\ln^2i}{2q}+\frac{5}{8}q
	\]
	i.e.
	\[
	q^2-\frac{3}{4}+ \left(2q-\frac{3}{4q} \right)\ln i
	>\frac{4q}{3i}+\frac{5}{8}q
	\]
	because
	\[
	\frac{2q^2-\frac{3}{4}}{2q^2}>\frac{1}{2q}.
	\]
	But $2q-\frac{3}{4q}>5/4$ and 
         \[
         q^2-\frac{3}{4}+ \left(2q-\frac{3}{4q} \right)\ln i>q^2-\frac{3}{4}+\frac{5}{4}\ln 3
      >q^2+\frac{1}{2}.
         \]
Also, $\frac{4q}{3i}+\frac{5}{8}q<\frac{4q}{9}+\frac{5}{8}q=\frac{77}{72}q$ and since
	$q^2+\frac{1}{2}>\frac{77}{72}q$ the lemma is proved.	
	\end{proof}
		
	\begin{lemma} \label{lem11}
For $p\geq 2$, $\frac{1}{p}+\frac{1}{q}=1$ there is a constant $c=c(p)>0$ such that for every
natural $n$  the next inequality is true:
\begin{align*}\label{eq11}
\frac{\ln^22+2q^2}{2^{1/q}}+\frac{2}{3}\frac{q}{2^{3/q}}-\ln^22-2q\sum_{k=2}^n\frac{1}{k^{1+2/q}}
-\frac{\ln^22-q\ln2+3/2q^2}{q2^{1+2/q}}>c(p).
\end{align*}
	\end{lemma}	
	\begin{proof}
	We have
	\[
	\sum_{k=2}^n\frac{1}{k^{1+2/q}}=\frac{1}{2^{1+2/q}}+\sum_{k=3}^n\frac{1}{k^{1+2/q}}
	<\frac{1}{2^{1+2/q}}+\int_2^\infty \frac{dx}{x^{1+2/q}}=\frac{1+q}{2.2^{2/q}}
	\]
and
		\[
	\frac{\ln^22-q\ln2+3/2q^2}{q2^{1+2/q}}
	=\frac{\left(\ln2-q/2\right)^2+5/4q^2}{q2^{1+2/q}}
	<\frac{1}{20q2^{2/q}}+\frac{5q}{2^{3+2/q}}.
	\]
	Since
	\[
	\frac{2}{3}\frac{q}{2^{3/q}}>\frac{1}{20q2^{2/q}}
	\]
	and
	\[
	\left(1-\frac{1}{2^{1/q}} \right)\ln^22<\frac{1}{4}
	\]
	it follows
	\begin{align*}
&\frac{\ln^22+2q^2}{2^{1/q}}+\frac{2}{3}\frac{q}{2^{3/q}}-\ln^22-2q\sum_{k=2}^n\frac{1}{k^{1+2/q}}
-\frac{\ln^22-q\ln2+3/2q^2}{q2^{1+2/q}}\\
&>\frac{q}{2^{2/q}}\left[\left(2^{1+1/q}-1 \right)q-\frac{13}{8} \right]-\frac{1}{4}
>\frac{1}{4}\left[\left(2^{1+1/q}-1 \right)q-\frac{21}{8} \right]>0
\end{align*}
since the function $f(x)=\left(2^{1+1/x}-1 \right)x$ is increasing and consequently\\ 
$\left(2^{1+1/q}-1 \right)q>f(1)=3$.
	\end{proof}

	\begin{lemma} \label{lem12}
For $p\geq 2$, $\frac{1}{p}+\frac{1}{q}=1$ and every natural numbers $i$ and $n$ such that
$2\le i\le n$  the next inequality is true:
\begin{align}\label{eq12}
\sum_{k=i}^{n}\frac{1}{k^p}&\left(k^{1/q}-\frac{1}{p} \right)^{p/q-1}
\left[k^{1/q}\left(\ln^2k-2q\ln k+2q^2\right)-2q^2\right]\notag\\
&>\frac{q\left(\ln^2i+2q^2 \right)}{i^{1/q}}-\frac{q\left(\ln^2n+2q^2 \right)}{n^{1/q}}
-2q^2\sum_{k=i}^{n}\frac{1}{k^{1+2/q}}\notag\\
&-\frac{\ln^2i-q\ln i+3/2q^2}{2i^{2/q}}+\frac{2q^2}{3i^{3/q}}-\frac{2q^2}{3n^{3/q}}.
\end{align}
	\end{lemma}	
\begin{proof}
\begin{align*}
\sum_{k=i}^{n}\frac{1}{k^p}&\left(k^{1/q}-\frac{1}{p} \right)^{p/q-1}
\left[k^{1/q}\left(\ln^2k-2q\ln k+2q^2\right)-2q^2\right]=J-I
\end{align*}
where
\begin{align*}
J&=\sum_{k=i}^{n}\frac{1}{k^p}\left(k^{1/q} \right)^{p/q-1}
\left[k^{1/q}\left(\ln^2k-2q\ln k+2q^2\right)-2q^2\right]\\
&=\sum_{k=i}^{n}\frac{\ln^2k-2q\ln k+2q^2}{k^{1+1/q}}-2q^2\sum_{k=i}^{n}\frac{1}{k^{1+2/q}}
\end{align*}
and
\begin{align*}
I&=\sum_{k=i}^{n}\frac{1}{k^p}
\left[\left(k^{1/q}\right)^{p/q-1}-\left(k^{1/q}-\frac{1}{p}- \right)^{p/q-1}\right]
\left[k^{1/q}\left(\ln^2k-2q\ln k+2q^2\right)-2q^2\right].\notag\\
\end{align*}
From \eqref{eq6} 
	\begin{align*}
J&>\frac{q\left(\ln^2i+2q^2 \right)}{i^{1/q}}
+\frac{\ln^2i-2q\ln i+2q^2}{2i^{1+1/q}}-\frac{q\left(\ln^2n+2q^2 \right)}{n^{1/q}}
-2q^2\sum_{k=i}^{n}\frac{1}{k^{1+2/q}}
\end{align*}
	and from \eqref{eq9} 
	\begin{align*}
I&<\frac{\ln^2i-2q\ln i+2q^2}{qi^{1+2/q}}+\frac{\ln^2i-q\ln i+3/2q^2}{2i^{2/q}}
-\frac{2q^2}{3i^{3/q}}+\frac{2q^2}{3n^{3/q}}.
\end{align*}
For $i\geq2$ we have  $qi^{1/q}\geq2$ and consequently	
	\begin{align*}
\frac{\ln^2i-2q\ln i+2q^2}{2i^{1+1/q}}\geq\frac{\ln^2i-2q\ln i+2q^2}{qi^{1+2/q}}.
\end{align*}
The lemma is proved.	
\end{proof}

	\begin{lemma} \label{lem13}
For $p\geq 2$,  $\frac{1}{p}+\frac{1}{q}=1$ and every natural $n$  the next inequality is true:
\begin{align}\label{eq13}
\sum_{k=1}^{n}\frac{1}{k^p}&\left(k^{1/q}-\frac{1}{p} \right)^{p/q-1}
\left[k^{1/q}\left(\ln^2k-2q\ln k+2q^2\right)-2q^2\right]\notag\\
&>\frac{q\left(\ln^22+2q^2 \right)}{2^{1/q}}-\frac{q\left(\ln^2n+2q^2 \right)}{n^{1/q}}
-2q^2\sum_{k=2}^{n}\frac{1}{k^{1+2/q}}\notag\\
&-\frac{\ln^22-q\ln 2+3/2q^2}{2^{1+2/q}}+\frac{2}{3}\frac{q^2}{2^{3/q}}-\frac{2q^2}{3n^{3/q}}.
\end{align}
	\end{lemma}	
\begin{proof}
Since
\begin{align*}
&\sum_{k=1}^{n}\frac{1}{k^p}\left(k^{1/q}-\frac{1}{p} \right)^{p/q-1}
\left[k^{1/q}\left(\ln^2k-2q\ln k+2q^2\right)-2q^2\right]\\
&=\sum_{k=2}^{n}\frac{1}{k^p}\left(k^{1/q}-\frac{1}{p} \right)^{p/q-1}
\left[k^{1/q}\left(\ln^2k-2q\ln k+2q^2\right)-2q^2\right]
\end{align*}
the lemma follows from the previous one.
\end{proof}

\begin{lemma} \label{lem15}
Let $p\geq 2$, $\frac{1}{p}+\frac{1}{q}=1$ and $A>1$. Then for
every natural numbers $i$ and $n$ such that
$i\le n$  the next equality is true:
\begin{align}\label{eq15.1}
&\sum_{j=2}^\infty (-1)^j \binom{p/q}{j}\frac{1}{A^j\ln^{2j}(n+1)}
\sum_{k=i}^n \frac{\ln^{2j}k-2qj\ln^{2j-1}k}{k^{1+1/q}}\notag\\
&=\frac{q}{i^{1/q}}\sum_{j=2}^\infty (-1)^j \binom{p/q}{j}\frac{\ln^{2j}i}{A^j\ln^{2j}(n+1)}
+O\left(\frac{1}{A^2\ln^{2}(n+1)i^{1/q}} \right)+O\left(\frac{1}{A^2n^{1/q}} \right).
\end{align}
	\end{lemma}	
\begin{proof}
Let us denote by $g$ the function
\[
g(x)=\frac{\ln^{2j}x-2qj\ln^{2j-1}x}{x^{1+1/q}}.
\]
Again, from Euler's summation formula 
\[
\sum_{k=i}^ng(k)=\int_i^ng(x)dx+\frac{g(i)+f(n)}{2}+\int_i^n\left(x-[x]-\frac{1}{2} \right)g'(x)dx
\]
where $[x]$ is the floor function. Now
\begin{equation*}
|g(i)|=\left|\frac{\ln^{2j}i-2qj\ln^{2j-1}i}{i^{1+1/q}} \right|<\frac{c(p)j\ln^{2j}i}{i^{1+1/q}}
<\frac{c(p)j^2\ln^{2j-2}n}{i^{1/q}}
\end{equation*}
and
\begin{equation*}
|g(n)|=\left|\frac{\ln^{2j}n-2qj\ln^{2j-1}n}{n^{1+1/q}} \right|<\frac{c(p)j\ln^{2j}n}{n^{1+1/q}}
<\frac{c(p)j^2\ln^{2j-2}n}{i^{1/q}}.
\end{equation*}
For $g(x)$ we have
\[
g'(x)=\frac{\ln^{2j-2}x}{x^{2+1/q}}
\left[-\left(1+\frac{1}{q}\right)\ln^2x+2j(q+2)\ln x-2qj(2j-1)  \right]
\]
and consequently
\[
|g'(x)|<\frac{c(p)j^2\ln^{2j}x}{x^{2+1/q}}<\frac{c(p)j^2\ln^{2j-2}x}{x^{1+1/q}}<\frac{c(p)j^2\ln^{2j-2}n}{x^{1+1/q}}.
\]
Then
\begin{equation*}
\left|\int_i^n\left(x-[x]-\frac{1}{2} \right)g'(x)dx\right|
<c(p)j^2\ln^{2j-2}n\int_i^n\frac{1}{x^{1+1/q}}<\frac{c(p)j^2\ln^{2j-2}n}{i^{1/q}}.
\end{equation*}
Also
\begin{equation*}
\int_i^ng(x)dx=\frac{q\ln^{2j}i}{i^{1/q}}-\frac{q\ln^{2j}n}{n^{1/q}}.
\end{equation*}
Consequently
\begin{align*}
\sum_{k=i}^n \frac{\ln^{2j}k-2qj\ln^{2j-1}k}{k^{1+1/q}}
=\frac{q\ln^{2j}i}{i^{1/q}}-\frac{q\ln^{2j}n}{n^{1/q}}+O\left(\frac{j^2\ln^{2j-2}n}{i^{1/q}} \right).
\end{align*}
Then
\begin{align*}
&\sum_{j=2}^\infty (-1)^j \binom{p/q}{j}\frac{1}{A^j\ln^{2j}(n+1)}
\sum_{k=i}^n \frac{\ln^{2j}k-2qj\ln^{2j-1}k}{k^{1+1/q}}\notag\\
&=\frac{q}{i^{1/q}}\sum_{j=2}^\infty (-1)^j \binom{p/q}{j}\frac{\ln^{2j}i}{A^j\ln^{2j}(n+1)}
+O\left(\frac{1}{A^2\ln^{2}(n+1)i^{1/q}} \right)+O\left(\frac{1}{A^2n^{1/q}} \right).
\end{align*}
\end{proof}

\section{Proof of the right inequality in theorem \ref{th1}
$$
d(a,b)\,\le\left(\frac{p}{p-1} \right)^p  \left(1+\frac{c}{\ln^2 \frac{b}{a}}\right)^{-1}\,,\qquad c=c(p)>0.
$$ }
By simple change of variables and notations it is easy to see that it is enough to prove \eqref{MR1} for the interval $(1,b)$.

From Holder's inequality we have for every two  functions $f(x)\geq0$ and $g(x)> 0, \,x\in(1,b)$ and $f^p(x)$
 and $g^q(x)$ are integrable over (1,b), 
\[
\left(\int_1^x f(t)dt \right)^p   \le\left(\int_1^x g^q(t)dt  \right)^{p/q}\left(\int_1^x \frac{f^p(t)}{g^p(t)}dt  \right).
\]
After multiplying both sides by $x^{-p}$, integrating from 1 to $b$ and changing the order of integration in the right side 
we get
\begin{align*}
\int_1^b\left(\frac{1}{x}\int_1^x f(t)dt \right)^pdx   \le 
\int_1^b\left[\frac{1}{g^p(t)}\int_t^b\left(\int_1^x g^q(u)du \right)^{p/q}\frac{dx}{x^p}  \right]f^p(t)dt.
\end{align*}
Let us denote for brevity $M(g,t)=g^{-p}(t)M^*(g,t)$ where
\[
M^*(g,t)=\int_t^b\left(\int_1^x g^q(u)du \right)^{p/q}\frac{dx}{x^p}.
\]
Then for every two functions $f(x)\geq 0$ and $g(x)> 0$, $1<x<b$ such that $f^p(x)$ and $g^q(x)$ are integrable the next upper estimation holds
\begin{align*}
\int_1^b\left(\frac{1}{x}\int_1^x f(t)dt \right)^pdx   \le 
\max_{1< t< b}M(g,t)  \int_1^b f^p(t)dt
\end{align*}
and consequently for every function $g(x)> 0,\,1<x<b$
\[
d(1,b)\le \max_{1< t< b}M(g,t). 
\]
Now we want to minimize 
\[
\max_{1< t< b}M(g,t)
\]
 over all functions  $g(x)> 0$ on the interval $(1,b)$ or to find 
\[
\min_{g(x)> 0}\,\max_{1< t< b}\frac{1}{g^p(t)}\int_t^b\left(\int_1^x g^q(u)du \right)^{p/q}\frac{dx}{x^p}.
\]
\begin{remark}
For $g(x)=x^{-1/(pq)}$ we obtain the original Hardy inequality. Indeed, we have
\[
\int_1^x g^q(u)du=qx^{1/q}-q<qx^{1/q},
\]
\[
\int_t^b\left(\int_1^x g^q(u)du \right)^{p/q}\frac{dx}{x^p}
<\int_t^b\left(qx^{1/q} \right)^{p/q}\frac{dx}{x^p}=q^p\left(t^{-1/q}-b^{-1/q} \right)<q^pt^{-1/q}
\]
for every $1< t< b$. Consequently $M(g,t)<q^p$ for every $1< t< b$, which means that
\[
\max_{1< t< b}M(g,t)<q^p \quad\mbox{i.e.}\quad d(1,b)\le q^p.
\]
\end{remark}

Now, for the function $g(x)$, defined by \eqref{eq200} we have
\begin{align*}
\int_1^xg^q(u)du&=\frac{q}{1+\alpha^2q^2}\left[x^{1/q}(\cos(\alpha\ln x)+\alpha q\sin(\alpha\ln x))-1 \right]\\
&<\frac{qx^{1/q}}{1+\alpha^2q^2}\left[\cos(\alpha\ln x)+\alpha q\sin(\alpha\ln x) \right]
\end{align*}
and for every $1< t< b$
\begin{equation*}\label{eqI1}
M^*(g,t)<q^{p/q}\int_t^b\left(\frac{\cos(\alpha\ln x)+\alpha q\sin(\alpha\ln x)}{1+\alpha^2q^2} \right)^{p/q}
\frac{dx}{x^{1+1/q}}.
\end{equation*}
Then from  \eqref{eqI31} it follows that for every $1< t< b$
\[
M^*(g,t)<-q^p\left(1+\frac{c}{\ln^2b}\right)^{-1}\int_t^b\left(g^p(x) \right)'dx
\le q^p\left(1+\frac{c}{\ln^2b}\right)^{-1}g^p(t)
\]
and consequently 
\[
M(g,t)\le  q^p\left(1+\frac{c}{\ln^2b}\right)^{-1}
=\left(\frac{p}{p-1} \right)^p  \left(1+\frac{c}{\ln^2 \frac{b}{a}}\right)^{-1}.
\]
The last means that 
\[
\max_{1< t< b}M(g,t)\le 
\left(\frac{p}{p-1} \right)^p  \left(1+\frac{c}{\ln^2 \frac{b}{a}}\right)^{-1}
\]
i.e.
\[
 d(1,b)\le \left(\frac{p}{p-1} \right)^p  \left(1+\frac{c}{\ln^2 \frac{b}{a}}\right)^{-1}.
\]

\section{Proof of the left inequality in theorem \ref{th1}
$$
d(a,b)\,\geq\left(\frac{p}{p-1} \right)^p-\frac{c}{\ln^2 \frac{b}{a}}\,,\qquad c=c(p)>0.
$$ }
By changing the order of integration, we write the left side of \eqref{eqI15} for
 $a=1$ and $f(x)>0,\,\,1<x<b$ in the following way
\begin{align*}
\int_1^b\left(\frac{1}{x}\int_1^xf(t)dt\right)^p \,dx=\int_1^b M(t)f^p(t)dt
\end{align*}
where
\begin{align*}
M(t)=\frac{1}{[f(t)]^{p/q}}\int_t^b\left(\int_1^x f(u)du \right)^{p/q}\frac{dx}{x^p}. 
\end{align*}
Obviously
\[
d(1,b)\geq \min_{1< t< b}M(t).
\]
Then for the function $f^*(x)$ defined in \eqref{eq21} we have
\[
\int_1^x f^*(u)du=qx^{1/q}\sin(\alpha\ln x)
\]
and
\begin{align}\label{eqI7}
\int_t^b\left(\int_1^x f^*(u)du \right)^{p/q}\frac{dx}{x^p}
=q^{p/q}\int_t^b (\sin(\alpha\ln x))^{p/q}\frac{dx}{x^{1+1/q}}.
\end{align}
Now let $0<\epsilon<1$.
Then  from \eqref{eqI7} and \eqref{eqI5} it follows that for $b>b_0$ where
 \[
b_0=e^{\pi/\sqrt{\min\{q(p-q)^{-1}(pq+1)^{-2},4(pq)^{-2} \}\epsilon}}
\]
the next inequality holds
\[
\int_t^b\left(\int_1^x f^*(u)du \right)^{p/q}\frac{dx}{x^p}
\geq -\frac{q^p}{1+(pq+\epsilon)\alpha^2 }\int_t^b\left[(f^*(x))^{p/q} \right]'dx
=\frac{q^p[f^*(t)]^{p/q}}{1+(pq+\epsilon)\alpha^2 }.
\]
Consequently for every $b>b_0$
\begin{align*}
M(t)\geq \frac{q^p}{1+(pq+\epsilon)\alpha^2 }\geq q^p\left(1-(pq+\epsilon)\alpha^2 \right)
\geq q^p-\frac{q^p(pq+\epsilon)\pi^2}{\ln^2b}
\end{align*}
i.e.
\begin{equation*}\label{I17}
d(1,b)\geq q^p-\frac{q^p(pq+\epsilon)\pi^2}{\ln^2b}.
\end{equation*}
Since for $b\le b_0$ we have
\[
d(1,b)\geq q^p-\frac{q^p\ln^2b_0}{\ln^2b}
\]
by taking $c=\max\{q^p(pq+\epsilon)\pi^2, q^p\ln^2b_0\}$ we complete the proof.

\section{Proof of the left inequality in theorem \ref{th2}
$$
d_n\geq\left(\frac{p}{p-1} \right)^p-\frac{c}{\ln^2 n}\,,\qquad c=c(p)>0.
$$ }
By changing the order of summation we write the left side of \eqref{eqI16} in the following way
\begin{equation*}
\sum_{k=1}^{n}\left(\frac{1}{k}\sum_{j=1}^{k}a_j\right)^p=
\sum_{i=1}^{n} \left[\frac{1}{a_i^{p/q}}  \sum_{k=i}^{n}\frac{1}{k^p}\left(\sum_{j=1}^{k}a_j\right)^{p/q}\right]a_i^p
=\sum_{i=1}^{n}M_ia_i^p
\end{equation*}
where
\[
M_i=\frac{1}{a_i^{p/q}} M_i^* \quad\mbox{and}\quad M_i^*=\sum_{k=i}^{n}\frac{1}{k^p}\left(\sum_{j=1}^{k}a_j\right)^{p/q}.
\]
Then
\begin{equation*}
\sum_{k=1}^{n}\left(\frac{1}{k}\sum_{j=1}^{k}a_j\right)^p\geq \min_{1\le i\le n}M_i\sum_{k=1}^{n}a_i^p
\end{equation*}
and consequently
\begin{equation*}\label{eq3}
d_n\geq \min_{1\le i\le n}M_i.
\end{equation*}
Now, we will prove that the sequence $a_k^*$ defined in Theorem \ref{th2} is the ``almost extremal'' sequence, i. e. 
the inequality \eqref{eqR1} holds.
We have
\[
\sum_{j=1}^{k}a_j^*=\int_1^{k+1}f^*(x)dx=q(k+1)^{1/q}\sin(\alpha\ln (k+1))
\]
where $f^*$ is the function defined by \eqref{eq20}. It is easy to see that the function 
$x^{p-1-1/q}\left(\sin(\alpha\ln x)\right)^{p/q}=\left(x^{1/q}\sin(\alpha\ln x)\right)^{p/q}$ is increasing and consequently
\begin{align}\label{eqR2}
M_i^*&=q^{p/q}\sum_{k=i}^{n}\frac{(k+1)^{p-1-1/q}}{k^p}\left(\sin(\alpha\ln (k+1))\right)^{p/q}\notag\\
&\geq q^{p/q}\int_i^{n+1}(\sin(\alpha\ln x))^{p/q}\frac{dx}{x^{1+1/q}}.
\end{align}
Since the function $f^*(x)$ is continuous there exists a point $\eta_i\in [i,i+1]$ such that
$a_i=f^*(\eta_i)$. Now let $0<\epsilon<1$. Then from \eqref{eqR2} and Lemma \ref{lem20} it follows that for every integer 
\[
n>n_0=e^{\pi/\sqrt{\min\{q(p-q)^{-1}(pq+1)^{-2},4(pq)^{-2} \}\epsilon}}
\]
\begin{align*}
M_i^*\geq -\frac{q^p}{1+(pq+\epsilon)\alpha^2 }\int_{\eta_i}^b\left[(f^*(x))^{p/q} \right]'dx
=\frac{q^p\left[f^*(\eta_i)\right]^{p/q}}{1+(pq+\epsilon)\alpha^2 }
\end{align*}
and  consequently 
\begin{align*}
M_i&\geq \frac{q^p}{1+(pq+\epsilon)\alpha^2 }\geq q^p\left(1-(pq+\epsilon)\alpha^2 \right)
\geq q^p-\frac{q^p(pq+\epsilon)\pi^2}{\ln^2(n+1)}
\end{align*}
i.e.
\[
d_n\geq  q^p-\frac{q^p(pq+\epsilon)\pi^2}{\ln^2(n+1)}.
\]
Since for $n\le n_0$ we have
\[
d_n\geq q^p-\frac{q^p\ln^2n_0}{\ln^2n}
\]
by taking $c=\max\{q^p(pq+\epsilon)\pi^2, q^p\ln^2n_0\}$ we complete the proof.

\section{Proof of the right inequality in theorem \ref{th2}
$$
d_n<\left(\frac{p}{p-1} \right)^p-\frac{c}{\ln^2 n}\,,\qquad c=c(p)>0.
$$ }
From Holder's inequality we have for every two  sequences $\mu_i>0$ and $\eta_i\geq0$, $i=1,...n$
\[
\sum_{i=1}^{k}\mu_i\eta_i\le\left(\sum_{i=1}^{k}\eta_i^p\right)^{1/p}\left(\sum_{i=1}^{k}\mu_i^q\right)^{1/q}
\]
or
\[
\left(\frac{1}{k}\sum_{i=1}^{k}\mu_i\eta_i\right)^p 
\le \frac{1}{k^p}\left(\sum_{i=1}^{k}\eta_i^p\right)\left(\sum_{i=1}^{k}\mu_i^q\right)^{p/q}.
\]
Denoting $a_i=\mu_i\eta_i$ 
and after changing the order of summation we get
\[
\sum_{k=1}^{n}\left(\frac{1}{k}\sum_{i=1}^{k}a_i\right)^p
\le \sum_{i=1}^{n}M_ia_i^p \le \left( \max_{1\le i\le n}M_i\right) \sum_{i=1}^{n}a_i^p,
\]
where 
\[
M_i=\frac{1}{\mu_i^p}M_i^*,\quad M_i^*=\sum_{k=i}^{n}\frac{1}{k^p}\left(\sum_{j=1}^{k}\mu_j^q\right)^{p/q}.
\]
Obviously 
\[
d_n\le \max_{1\le i\le n}M_i,\quad\mbox{so we want to minimize }\quad
\max_{1\le i\le n}M_i
\]
 over all sequences $\mu=\{\mu_i>0\},\,i=1,2, ... ,n$, i.e. to find 
\[
\min_{\mu>0}\,\max_{1\le i\le n}M_i
\]
or, at least, to make it as small as possible.
\begin{remark}
By choosing, for instance, 
\[
\mu_k= k^{-1/(pq)},\quad k=1,2,\,...\,n                  
\]
 we obtain the Hardy's inequality with $d_n=\left(\frac{p}{p-1} \right)^p$.

Indeed,
\begin{equation}\label{eq14}
\sum_{j=1}^{k}\mu_j^q=1+\sum_{j=2}^{k}\frac{1}{j^{1/p}}<1+\int_1^k\frac{dx}{x^{1/p}}=qk^{1/q}-q+1=
qk^{1/q}\left(1-\frac{1}{pk^{1/q}} \right)
\end{equation}
and from \eqref{eq5} of Lemma \ref{lem05}
\[
M_i^*\le q^{p/q}\sum_{k=i}^{n}\frac{1}{k^{1+1/q}}\left(1-\frac{1}{pk^{1/q}}\right)^{p/q}
\le\left(\frac{p}{p-1} \right)^pi^{-1/q}.	
\]
Consequently
\[
M_i\le\left(\frac{p}{p-1} \right)^p \quad\mbox{and}\quad d_n\le\left(\frac{p}{p-1} \right)^p.
\]
\end{remark}
But in order to prove the right inequality of \eqref{MR2} we need to make a more complicated choise of the sequence $\mu_k$.
Let
\[
\mu_k=
\left(\frac{A}{k^{1/p}}-\frac{1}{\ln^2(n+1)}\int_{k}^{k+1}\frac{\ln^2 x}{x^{1/p}}dx\right)^{1/q}
\]
where $A=A(p)>2$ is a constant which depends only on $p$ and will be chosen later.
It is obvious that the sequence $\mu_k,\, k=1,...,n$ is well defined.
Then for every $i$ such that $1\le i\le n$
\begin{equation}\label{eq31}
\mu_i^p<\frac{A^{p/q}}{i^{1/q}}=\frac{c(p)}{i^{1/q}}
\end{equation}
and
\begin{align}\label{eq15}
\mu_i^p&>\left(\frac{A}{i^{1/p}}-\frac{\ln^2(i+1)}{i^{1/p}\ln^2(n+1)}\right)^{p/q}
=\frac{A^{p/q}}{i^{1/q}}\left(1-\frac{\ln^2(i+1)}{A\ln^2(n+1)}\right)^{p/q}\notag\\
&=\frac{A^{p/q}}{i^{1/q}}\sum_{j=0}^\infty(-1)^j \binom{p/q}{j}
\left(\frac{\ln^2(i+1)}{A\ln^2(n+1)} \right)^j.
\end{align}
Now
\begin{align*}
\sum_{j=1}^{k}\mu_j^q
&=A\sum_{j=1}^{k}\frac{1}{j^{1/p}}-\frac{1}{\ln^2(n+1)}\int_{1}^{k+1}\frac{\ln^2 x}{x^{1/p}}dx\\
&<A\sum_{j=1}^{k}\frac{1}{j^{1/p}}-\frac{1}{\ln^2(n+1)}\int_{1}^{k}\frac{\ln^2 x}{x^{1/p}}dx.
\end{align*}
Since
\[
\int_{1}^{k}\frac{\ln^2 x}{x^{1/p}}dx=
qk^{1/q}\left(\ln^2k-2q\ln k+2q^2 \right)-2q^3
\]
and from \eqref{eq14} we have
\begin{align*}
\sum_{j=1}^{k}\mu_j^q
&<Aq\left(k^{1/q}-\frac{1}{p} \right)-\frac{q}{\ln^2(n+1)}\left[k^{1/q}\left(\ln^2k-2q\ln k+2q^2 \right)-2q^2 \right]\\
&=Aq\left(k^{1/q}-\frac{1}{p} \right)(1-S(k))
\end{align*}
where for brevity we denoted by
\[
S(k)=\frac{k^{1/q}\left(\ln^2k-2q\ln k+2q^2 \right)-2q^2}{A\left(k^{1/q}-\frac{1}{p} \right)\ln^2(n+1)}.
\]
We have
\begin{align}\label{eq17}
(Aq)^{-p/q}M_i^*&<(Aq)^{-p/q}\sum_{k=i}^{n} \frac{1}{k^p}\left[Aq\left(k^{1/q}-\frac{1}{p} \right)\right]^{p/q}
\left(1-S(k)\right)^{p/q}\notag\\
&=\sum_{k=i}^n \frac{1}{k^p}\left(k^{1/q}-\frac{1}{p} \right)^{p/q}
\sum_{j=0}^\infty (-1)^j \binom{p/q}{j}S^j(k) \notag\\
&=\sum_{k=i}^n \frac{1}{k^p}\left(k^{1/q}-\frac{1}{p} \right)^{p/q}
-\frac{p}{q}\sum_{k=i}^n \frac{1}{k^p}\left(k^{1/q}-\frac{1}{p} \right)^{p/q}S(k) \notag\\
&+\sum_{k=i}^n \frac{1}{k^p}\left(k^{1/q}-\frac{1}{p} \right)^{p/q}
\sum_{j=2}^\infty (-1)^j \binom{p/q}{j}S^j(k)=L_1-L_2+L_3.
\end{align}
From \ref{eq5} of Lemma \ref{lem05}
\begin{equation}\label{L1}
L_1=\sum_{k=i}^n \frac{1}{k^p}\left(k^{1/q}-\frac{1}{p} \right)^{p/q}
\le q\left(\frac{1}{i^{1/q}}- \frac{1}{(n+1)^{1/q}}\right).
\end{equation}

From \ref{eq12} of Lemma \ref{lem12} we have for $2\le i\le n$
\begin{align}\label{L21}
L_2(i\geq2)
&=\frac{p}{q}\sum_{k=i}^n \frac{1}{k^p}\left(k^{1/q}-\frac{1}{p} \right)^{p/q}S(k)\notag\\
&>\frac{p}{qA\ln^2(n+1)}
\Big[\frac{q\left(\ln^2i+2q^2 \right)}{i^{1/q}}-\frac{q\left(\ln^2n+2q^2 \right)}{n^{1/q}}
-2q^2\sum_{k=i}^{n}\frac{1}{k^{1+2/q}}\notag\\
&-\frac{\ln^2i-q\ln i+3q^2/2}{2i^{2/q}}+\frac{2q^2}{3i^{3/q}}-\frac{2q^2}{3n^{3/q}}\Big]
\end{align}
and since $S(1)=0$
\begin{align}\label{L22}
L_2(i=1)
&=\frac{p}{q}\sum_{k=1}^n \frac{1}{k^p}\left(k^{1/q}-\frac{1}{p} \right)^{p/q}S(k)\notag\\
&>\frac{p}{qA\ln^2(n+1)}
\Big[\frac{q\left(\ln^22+2q^2 \right)}{2^{1/q}}-\frac{q\left(\ln^2n+2q^2 \right)}{n^{1/q}}
-2q^2\sum_{k=2}^{n}\frac{1}{k^{1+2/q}}\notag\\
&-\frac{\ln^22-q\ln 2+3q^2/2}{2^{1+2/q}}+\frac{2}{3}\frac{q^2}{2^{3/q}}-\frac{2q^2}{3n^{3/q}}\Big]
\end{align}
for $i=1$.

For $k\geq2$ from \eqref{eq1} of Lemma \ref{lem01} and since
\[
\left|\frac{2q}{\ln k}-\frac{2q^2-2q^2k^{-1/q} }{\ln^2k}\right|<\frac{c(p)}{\ln k}
\]
it follows that for every natural $j\geq 2$
\begin{align*}
S^j(k)&=\left[\frac{k^{1/q}\ln^2k}{A\left(k^{1/q}-\frac{1}{p} \right)\ln^2(n+1)} \right]^j
\left[1-\frac{2qj}{\ln k}+\frac{\left(2q^2-2q^2k^{-1/q} \right)j}{\ln^2k}+O\left(\frac{j^2}{\ln^2k} \right) \right]\\
&=\left[\frac{k^{1/q}\ln^2k}{A\left(k^{1/q}-\frac{1}{p} \right)\ln^2(n+1)} \right]^j
\left[1-\frac{2qj}{\ln k}+O\left(\frac{j^2}{\ln^2k} \right) \right]\\
&=\left[\frac{k^{1/q}}{A\left(k^{1/q}-\frac{1}{p} \right)\ln^2(n+1)} \right]^j
\left[\ln^{2j}k-2qj\ln^{2j-1}k \right]\\
&+\left[\frac{k^{1/q}\ln^2k}{A\left(k^{1/q}-\frac{1}{p} \right)\ln^2(n+1)} \right]^j
O\left(\frac{j^2}{\ln^2k} \right)\\
&=\left[\frac{k^{1/q}}{A\left(k^{1/q}-\frac{1}{p} \right)\ln^2(n+1)} \right]^j
\left[\ln^{2j}k-2qj\ln^{2j-1}k \right]+\left(\frac{2}{A}\right)^jO\left(\frac{j^2}{\ln^2(n+1)} \right).
\end{align*}
Then for $i\geq 2$
\begin{align*}
L_3&=\sum_{k=i}^n \frac{1}{k^p}\left(k^{1/q}-\frac{1}{p} \right)^{p/q}
\sum_{j=2}^\infty (-1)^j \binom{p/q}{j}S^j(k)\\
&=\sum_{k=i}^n \frac{1}{k^p}\left(k^{1/q}-\frac{1}{p} \right)^{p/q}
\sum_{j=2}^\infty  (-1)^j\binom{p/q}{j}
\frac{k^{j/q}\left(\ln^{2j}k-2qj\ln^{2j-1}k\right)}{A^j\left(k^{1/q}-\frac{1}{p} \right)^j\ln^{2j}(n+1)} \\
&+\sum_{k=i}^n \frac{1}{k^p}\left(k^{1/q}-\frac{1}{p} \right)^{p/q}
\sum_{j=2}^\infty (-1)^j \binom{p/q}{j}\left(\frac{2}{A}\right)^jO\left(\frac{j^2}{\ln^2(n+1)} \right)
=L_{31}+L_{32}.
\end{align*}
Now
\begin{align}\label{eq27}
\left|L_{32}\right|&=\left|\sum_{k=i}^n \frac{1}{k^p}\left(k^{1/q}-\frac{1}{p} \right)^{p/q}
\sum_{j=2}^\infty (-1)^j \binom{p/q}{j}\left(\frac{2}{A}\right)^jO\left(\frac{j^2}{\ln^2(n+1)} \right)\right|\notag\\
&<\frac{c(p)}{\ln^2(n+1)}\sum_{k=i}^n \frac{1}{k^p}\left(k^{1/q} \right)^{p/q}
\sum_{j=2}^\infty \left|\binom{p/q}{j}\right|\left(\frac{2}{A}\right)^j j^2\notag\\
&=\frac{c(p)}{A^2\ln^2(n+1)}\sum_{k=i}^n\frac{1}{k^{1+1/q}}
<\frac{c(p)}{i^{1/q}A^2\ln^2(n+1)}=\frac{1}{i^{1/q}}O\left(\frac{1}{A^2\ln^2(n+1)} \right).
\end{align}
For $\alpha\geq 1$ and $0\le x\le 1/2$ the next inequality holds
\[
\frac{1}{(1-x)^\alpha}\le 1+\alpha 2^{\alpha+1}x
\]
which is easy to verify. Then for $p\geq 2$ and $j\geq 1$ we have
\[
\frac{k^{j/q}}{\left(k^{1/q}-\frac{1}{p} \right)^j}=\frac{1}{\left(1-\frac{1}{pk^{1/q}} \right)^j}
=1+O\left(\frac{2^j j}{k^{1/q}} \right)
=1+O\left(\frac{2^j j}{\ln^2k} \right).
\]
Consequently
\begin{align*}
&\sum_{j=2}^\infty (-1)^j \binom{p/q}{j}
\frac{k^{j/q}\left(\ln^{2j}k-2qj\ln^{2j-1}k\right)}{A^j\left(k^{1/q}-\frac{1}{p} \right)^j\ln^{2j}(n+1)}\notag \\
&=\sum_{j=2}^\infty (-1)^j \binom{p/q}{j}
\frac{\ln^{2j}k-2qj\ln^{2j-1}k}{A^j\ln^{2j}(n+1)}+
O\left(\frac{1}{A^2\ln^2(n+1)}\right).
\end{align*}
Then by \eqref{eq15.1} of Lemma \ref{lem15}
\begin{align}\label{eq29}
L_{31}&=\sum_{k=i}^n \frac{1}{k^p}\left(k^{1/q}-\frac{1}{p} \right)^{p/q}
\sum_{j=2}^\infty (-1)^j \binom{p/q}{j}
\frac{\ln^{2j}k-2qj\ln^{2j-1}k}{A^j\ln^{2j}(n+1)}\notag\\
&+\sum_{k=i}^n \frac{1}{k^p}\left(k^{1/q}-\frac{1}{p} \right)^{p/q}
O\left(\frac{1}{A^2\ln^2(n+1)}\right)\notag\\
&=\sum_{k=i}^n \frac{1}{k^p}\left(k^{1/q}-\frac{1}{p} \right)^{p/q}
\sum_{j=2}^\infty (-1)^j \binom{p/q}{j}
\frac{\ln^{2j}k-2qj\ln^{2j-1}k}{A^j\ln^{2j}(n+1)}\notag\\
&+\frac{1}{i^{1/q}}O\left(\frac{1}{A^2\ln^2(n+1)}\right)\notag\\
&=\sum_{k=i}^n \frac{1}{k^{1+1/q}}
\sum_{j=2}^\infty (-1)^j \binom{p/q}{j}
\frac{\ln^{2j}k-2qj\ln^{2j-1}k}{A^j\ln^{2j}(n+1)}\notag\\
&+\sum_{k=i}^n \frac{1}{k^{1+1/q}}
\sum_{j=2}^\infty (-1)^j \binom{p/q}{j}
\frac{\ln^{2j}k-2qj\ln^{2j-1}k}{A^j\ln^{2j}(n+1)}O\left(\frac{1}{k^{1/q}} \right)\notag\\
&+\frac{1}{i^{1/q}}O\left(\frac{1}{A^2\ln^2(n+1)}\right)\notag\\
&=\sum_{k=i}^n \frac{1}{k^{1+1/q}}
\sum_{j=2}^\infty (-1)^j \binom{p/q}{j}
\frac{\ln^{2j}k-2qj\ln^{2j-1}k}{A^j\ln^{2j}(n+1)}\notag\\
&+\frac{1}{i^{1/q}}O\left(\frac{1}{A^2\ln^2(n+1)}\right)\notag\\
&=\sum_{j=2}^\infty (-1)^j \binom{p/q}{j}\frac{1}{A^j\ln^{2j}(n+1)}
\sum_{k=i}^n \frac{\ln^{2j}k-2qj\ln^{2j-1}k}{k^{1+1/q}}\notag\\
&+\frac{1}{i^{1/q}}O\left(\frac{1}{A^2\ln^2(n+1)}\right)\notag\\
&=\frac{q}{i^{1/q}}\sum_{j=2}^\infty (-1)^j \binom{p/q}{j}\frac{\ln^{2j}i}{A^j\ln^{2j}(n+1)}
+\frac{1}{i^{1/q}}O\left(\frac{1}{A^2\ln^2(n+1)}\right)+O\left(\frac{1}{A^2 n^{1/q}}\right)
\end{align}
because
\begin{align*}
\frac{1}{k^p}\left(k^{1/q}-\frac{1}{p} \right)^{p/q}
&=\frac{1}{k^{1+1/q}}\left(1-\frac{1}{pk^{1/q}} \right)^{p/q}
=\frac{1}{k^{1+1/q}}\left(1+ O\left(\frac{1}{k^{1/q}} \right)\right),
\end{align*}

\begin{align*}
\left| \binom{p/q}{j}\frac{(-1)^j\left(\ln^{2j}k-2qj\ln^{2j-1}k\right)}{A^j\ln^{2j}(n+1)}O\left(\frac{1}{k^{1/q}} \right)\right|
<\frac{cj^{p/q+1}\ln^2k}{A^j\ln^{2}(n+1)k^{1/q}}
<\frac{cj^{p/q+1}}{A^j\ln^{2}(n+1)}
\end{align*}
and
\begin{align*}
\sum_{k=i}^n \frac{1}{k^{1+1/q}}\sum_{j=2}^\infty\frac{cj^{p/q+1}}{A^j\ln^{2}(n+1)}
<\frac{c}{A^2\ln^{2}(n+1)}\sum_{k=i}^n \frac{1}{k^{1+1/q}}
<\frac{c}{A^2\ln^2(n+1)i^{1/q}}.
\end{align*}
Consequently for $i\geq2$
\begin{align}\label{eq30}
L_{3}(i\geq2)=\frac{q}{i^{1/q}}\sum_{j=2}^\infty (-1)^j \binom{p/q}{j}\frac{\ln^{2j}i}{A^j\ln^{2j}(n+1)}
+\frac{1}{i^{1/q}}O\left(\frac{1}{A^2\ln^2(n+1)}\right)+O\left(\frac{1}{A^2 n^{1/q}}\right).
\end{align}
Since $S(1)=0$ we have for $i=1$
\begin{align}\label{eq31}
L_{3}(i=1)&=\sum_{k=1}^n \frac{1}{k^p}\left(k^{1/q}-\frac{1}{p} \right)^{p/q}
\sum_{j=2}^\infty (-1)^j \binom{p/q}{j}S^j(k)\notag\\
&=\sum_{k=2}^n \frac{1}{k^p}\left(k^{1/q}-\frac{1}{p} \right)^{p/q}
\sum_{j=2}^\infty (-1)^j \binom{p/q}{j}S^j(k)\notag\\
&=\frac{q}{2^{1/q}}\sum_{j=2}^\infty (-1)^j \binom{p/q}{j}\frac{\ln^{2j}2}{A^j\ln^{2j}(n+1)}
+O\left(\frac{1}{A^2\ln^2(n+1)}\right)+O\left(\frac{1}{A^2 n^{1/q}}\right)\notag\\
&=\frac{q}{2^{1/q}}\sum_{j=2}^\infty (-1)^j \binom{p/q}{j}\frac{\ln^{2j}2}{A^j\ln^{2j}(n+1)}
+O\left(\frac{1}{A^2\ln^2(n+1)}\right)
\end{align}

We consider the cases $i=1$ and $i\geq2$ separately.

\textbf{Case $i=1$.}\\
From \eqref{L1}, \eqref{L21} and  \eqref{eq31} we obtain
\begin{align*}
&(Aq)^{-p/q}M_1^*\\
&<q-\frac{p}{A\ln^2(n+1)}\left[\frac{\ln^22+2q^2}{2^{1/q}}+\frac{2}{3}\frac{q}{2^{3/q}}
-2q\sum_{k=2}^n\frac{1}{k^{1+2/q}}
-\frac{\ln^22-q\ln2+3q^2/2}{q2^{1+2/q}} \right]\\
&\,\,\,\,\,\,\,\,\,+\frac{q}{2^{1/q}}\sum_{j=2}^\infty (-1)^j \binom{p/q}{j}\frac{\ln^{2j}2}{A^j\ln^{2j}(n+1)}
+O\left(\frac{1}{A^2\ln^2(n+1)}\right)-T\\
&=q\sum_{j=0}^\infty (-1)^j \binom{p/q}{j}\frac{\ln^{2j}2}{A^j\ln^{2j}(n+1)}
-q\left(1-\frac{1}{2^{1/q}}\right)\sum_{j=2}^\infty (-1)^j \binom{p/q}{j}\frac{\ln^{2j}2}{A^j\ln^{2j}(n+1)}\\
&-\frac{p}{A\ln^2(n+1)}\left[\frac{\ln^22+2q^2}{2^{1/q}}+\frac{2}{3}\frac{q}{2^{3/q}}
-\ln^22-2q\sum_{k=2}^n\frac{1}{k^{1+2/q}}
-\frac{\ln^22-q\ln2+3/2q^2}{q2^{1+2/q}} \right]\\
&+O\left(\frac{1}{A^2\ln^2(n+1)}\right)-T\\
&=q\sum_{j=0}^\infty (-1)^j \binom{p/q}{j}\frac{\ln^{2j}2}{A^j\ln^{2j}(n+1)}
+O\left(\frac{1}{A^2\ln^2(n+1)}\right)-T\\
&-\frac{p}{A\ln^2(n+1)}\left[\frac{\ln^22+2q^2}{2^{1/q}}+\frac{2}{3}\frac{q}{2^{3/q}}
-\ln^22-2q\sum_{k=2}^n\frac{1}{k^{1+2/q}}
-\frac{\ln^22-q\ln2+3/2q^2}{q2^{1+2/q}} \right]\\
\end{align*}
where
\[
T=\frac{q}{(n+1)^{1/q}}
-\frac{p}{A\ln^2(n+1)} \left[\frac{\ln^2n+2q^2 }{n^{1/q}}
+\frac{2q}{3n^{3/q}}\right].
\]
It is obvious that by taking $A=A(p)$ big enough we can make $T$ positive.
By Lemma \ref{lem11} there is a constant $c(p)$ such that
\begin{align*}
\frac{\ln^22+2q^2}{2^{1/q}}+\frac{2}{3}\frac{q}{2^{3/q}}
-\ln^22-\sum_{k=2}^n\frac{2q}{k^{1+2/q}}
-\frac{\ln^22-q\ln2+3/2q^2}{q2^{1+2/q}}>c(p).
\end{align*} 
Then
\begin{align*}
&(Aq)^{-p/q}M_1^*<
q\sum_{j=0}^\infty (-1)^j \binom{p/q}{j}\frac{\ln^{2j}2}{A^j\ln^{2j}(n+1)}
+O\left(\frac{1}{A^2\ln^2(n+1)}\right)-\frac{pc(p)}{A\ln^2(n+1)}.
\end{align*}
Again, by taking $A=A(p)$ big enough we have
\begin{align*}
&(Aq)^{-p/q}M_1^*<
q\sum_{j=0}^\infty (-1)^j \binom{p/q}{j}\frac{\ln^{2j}2}{A^j\ln^{2j}(n+1)}
-\frac{c(p)}{A\ln^2(n+1)}
\end{align*}
and from \eqref{eq15}  and \eqref{eq31} we get
\[
M_1<\left(\frac{p}{p-1} \right)^p-\frac{c(A,p)}{\ln^2(n+1)}
=\left(\frac{p}{p-1} \right)^p-\frac{c(p)}{\ln^2(n+1)}.
\]

\textbf{Case $i\geq2$.}\\
From \eqref{L1}, \eqref{L21}, \eqref{eq30} we obtain
\begin{align*}\label{eq18}
&(Aq)^{-p/q}M_i^*
<q\left[\frac{1}{i^{1/q}}-\frac{1}{(n+1)^{1/q}}\right]\\
&-\frac{p}{qA\ln^2(n+1)}
\Big[\frac{q\left(\ln^2i+2q^2 \right)}{i^{1/q}}-\frac{q\left(\ln^2n+2q^2 \right)}{n^{1/q}}
-2q^2\sum_{k=i}^{n}\frac{1}{k^{1+2/q}}\notag\\
&-\frac{\ln^2i-q\ln i+3/2q^2}{2i^{2/q}}+\frac{2q^2}{3i^{3/q}}-\frac{2q^2}{3n^{3/q}}\Big]
+O\left(\frac{1}{A^2n^{1/q}} \right)\\
&+\frac{q}{i^{1/q}}\sum_{j=2}^\infty (-1)^j \binom{p/q}{j}\frac{\ln^{2j}i}{A^j\ln^{2j}(n+1)}
+\frac{1}{i^{1/q}}O\left(\frac{1}{A^2\ln^2(n+1)}\right)\\
&=\frac{q}{i^{1/q}}\sum_{j=0}^\infty (-1)^j \binom{p/q}{j}\frac{\ln^{2j}i}{A^j\ln^{2j}(n+1)}
-\frac{q}{(n+1)^{1/q}}\\
&-\frac{p}{A\ln^2(n+1)}
\Big[\frac{2q^2}{i^{1/q}}-\frac{\ln^2n+2q^2}{n^{1/q}}
-2q\sum_{k=i}^{n}\frac{1}{k^{1+2/q}}+\frac{2q}{3i^{3/q}}-\frac{2q}{3n^{3/q}}\notag\\
&-\frac{\ln^2i-q\ln i+3/2q^2}{2qi^{2/q}}\Big]
+\frac{1}{i^{1/q}}O\left(\frac{1}{A^2\ln^2(n+1)}\right)+O\left(\frac{1}{A^2n^{1/q}} \right)\\
&=\frac{q}{i^{1/q}}\sum_{j=0}^\infty (-1)^j \binom{p/q}{j}\frac{\ln^{2j}(i+1)}{A^j\ln^{2j}(n+1)}
-\frac{q}{i^{1/q}}\sum_{j=2}^\infty (-1)^j \binom{p/q}{j}\frac{\ln^{2j}(i+1)-\ln^{2j}i}{A^j\ln^{2j}(n+1)}\notag\\
&-\frac{p}{A\ln^2(n+1)}
\Big[\frac{2q^2-\ln^2(i+1)+\ln^2i}{i^{1/q}}
-2q\sum_{k=i}^{n}\frac{1}{k^{1+2/q}}\notag\\
&-\frac{\ln^2i-q\ln i+3/2q^2}{2qi^{2/q}}+\frac{2q}{3i^{3/q}}\Big]+
\frac{1}{i^{1/q}}O\left(\frac{1}{A^2\ln^2(n+1)}\right)-T
\end{align*}
where
\[
T=\frac{q}{(n+1)^{1/q}}
-\frac{p}{A\ln^2(n+1)} \left[\frac{\ln^2n+2q^2 }{n^{1/q}}
+\frac{2q}{3n^{3/q}}\right]+O\left(\frac{1}{A^2n^{1/q}} \right).
\]
By taking $A=A(p)$ big enough we can make $T$  positive.
Now
\begin{align*}
\left|\frac{q}{i^{1/q}}\sum_{j=2}^\infty (-1)^j \binom{p/q}{j}\frac{\ln^{2j}(i+1)-\ln^{2j}i}{A^j\ln^{2j}(n+1)}\right|
=\frac{1}{i^{1/q}}O\left(\frac{1}{A^2\ln^2(n+1)} \right).
\end{align*}
Also we have 
\[
\ln^2(i+1)-\ln^2i<\frac{3}{4}
\]
and
\[
\sum_{k=i}^{n}\frac{1}{k^{1+2/q}}\le \frac{1}{i^{1+2/q}}+\int_i^n\frac{dx}{x^{1+2/q}}
=\frac{1}{i^{1+2/q}}+\frac{q}{2i^{2/q}}
\]
and consequently
\begin{align*}
&\frac{2q^2-\ln^2(i+1)+\ln^2i}{i^{1/q}}
-2q\sum_{k=i}^{n}\frac{1}{k^{1+2/q}}
-\frac{\ln^2i-q\ln i+3/2q^2}{2qi^{2/q}}+\frac{2q}{3i^{3/q}}\\
&>\frac{1}{i^{1/q}}\left[2q^2-\frac{3}{4}+\frac{2q}{3i^{2/q}}-\frac{2q}{i^{1+1/q}}-\frac{q^2}{i^{1/q}}
-\frac{\ln^2i-q\ln i+3/2q^2}{2qi^{1/q}}\right].
\end{align*}
By Lemma \ref{lem10} there is a constant $c(p)>0$ such that
\begin{align*}
\frac{2q^2-\ln^2(i+1)+\ln^2i}{i^{1/q}}
-2q\sum_{k=i}^{n}\frac{1}{k^{1+2/q}}
-\frac{\ln^2i-q\ln i+3/2q^2}{2qi^{2/q}}+\frac{2q}{3i^{3/q}}>\frac{c(p)}{i^{1/q}}.
\end{align*}

Then
\begin{align*}
&(Aq)^{-p/q}M_i^*\\
&<\frac{q}{i^{1/q}}\sum_{j=0}^\infty (-1)^j \binom{p/q}{j}\frac{\ln^{2j}(i+1)}{A^j\ln^{2j}(n+1)}
-\frac{pc(p)}{A\ln^2(n+1)i^{1/q}}+\frac{1}{i^{1/q}}O\left(\frac{1}{A^2\ln^2(n+1)} \right).
\end{align*}
By taking $A=A(p)$ big enough we obtain
\[
M_i^*<\frac{A^{p/q}q^p}{i^{1/q}}\sum_{j=0}^\infty (-1)^j \binom{p/q}{j}\frac{\ln^{2j}(i+1)}{A^j\ln^{2j}(n+1)}
-\frac{c(A,p)}{i^{1/q}\ln^2(n+1)}
\]
and from \eqref{eq15} and \eqref{eq31} we get
\[
M_i<q^p-\frac{c(A,p)}{\ln^2(n+1)}
=\left(\frac{p}{p-1} \right)^p-\frac{c(A,p)}{\ln^2(n+1)}
=\left(\frac{p}{p-1} \right)^p-\frac{c(p)}{\ln^2(n+1)}.
\]

\bibliographystyle{amsplain}

\end{document}